\RequirePackage{fix-cm}
\documentclass[smallextended]{svjour3}       
\smartqed  
\usepackage{microtype}
\usepackage{graphicx}
\usepackage{mathptmx}      
\usepackage[tight-spacing=true]{siunitx}
\usepackage{array}
\usepackage{amsmath}
\usepackage{amsfonts}
\usepackage{mathtools}
\usepackage{algorithm}
\usepackage{algpseudocode}
\usepackage[subrefformat=parens]{subcaption}
\usepackage{hyperref}
\usepackage{doi}
\newcommand{\nullmat}{O}
\newcommand{\TP}{{T}}
\newcommand{\HC}{{H}}
\newcommand{\pp}{\phantom{+}}
\newcommand{\figscale}{0.46}
\newcommand{\figwidth}{0.49}

\begin{document}

\title{
A thick-restart Lanczos type method for Hermitian $J$-symmetric eigenvalue problems
}


\author{Ken-Ichi Ishikawa \and Tomohiro Sogabe}


\institute{K.-I. Ishikawa \at
Core of Research for the Energetic Universe, 
Graduate School of Science, 
Hiroshima University, 
Higashi-Hiroshima 739-8526, Japan \\
ORCID: 0000-0001-6425-1894 \\
Tel.: +81-82-424-7363, Fax:  +81-82-424-0717\\
\email{ishikawa@theo.phys.sci.hiroshima-u.ac.jp}
   \and
T. Sogabe \at
Department of Applied Physics, Nagoya University,
Furo-cho, Chikusa-ku, Nagoya 464-8603, Japan\\
E-mail: sogabe@na.nuap.nagoya-u.ac.jp
}

\date{Received: date / Accepted: date}

\maketitle

\begin{abstract}
A thick-restart Lanczos type algorithm is proposed for Hermitian $J$-sym\-metric matrices.
Since Hermitian $J$-symmetric matrices possess doubly degenerate spectra
or doubly multiple eigenvalues
with a simple relation between the degenerate eigenvectors,
we can improve the convergence of the Lanczos algorithm by restricting the search space 
of the Krylov subspace to that spanned by one of each pair of the degenerate eigenvector pairs.
We show that the Lanczos iteration is compatible with the $J$-symmetry, so that the subspace 
can be split into two subspaces that are orthogonal to each other. 
The proposed algorithm searches for eigenvectors in one of the two subspaces without the multiplicity.
The other eigenvectors paired to them can be easily reconstructed with the simple relation from the $J$-symmetry.
We test our algorithm on randomly generated small dense matrices and
a sparse large matrix originating from a quantum field theory.

\end{abstract}

\section{Introduction}
\label{intro}

In theoretical physics, symmetry is an important guiding principle to search for laws of nature.
When numerically simulating a physics system that has a symmetry, it is preferable to retain 
the symmetry property in the numerical simulation algorithm. 
Hermitian matrices often appear in analyzing physics systems,
and spectra analysis is important to understand the nature.
In this case the Hermitian symmetry must be considered in the spectra analysis and the eigensolver algorithm.
Developing eigensolver algorithms for spectra analysis is one of the major research areas in applied mathematics.

In physics, matrices having both Hermiticity symmetry and $J$-symmetry exist.
Let $A$ be a Hermitian and $J$-symmetric matrix in $\mathbb{C}^{n\times n}$.
The $J$-symmetric property of $A$ is defined by
\begin{align}
  J A J^{-1} &= A^{\TP}, \quad J^{\TP} = -J,\quad J^{\TP}=J^{-1},
\label{eq:symmetry}
\end{align}
where $J$ is a skew-symmetric matrix in $\mathbb{R}^{n\times n}$. 
We take this definition 
from~\cite{PETKOVIVANOV1994} for 
the $J$-symmetry~\footnote{The literature shows 
                           another definition, $(JA)=(JA)^T$, for the $J$-symmetry;
                           however, this is opposite from ours because $JAJ^{-1}=-A^{T}$.}.
Throughout this paper, we employ the following notations:
$^{\TP}$ for matrix transposition, 
$^{\HC}$ for Hermitian conjugation, and $^*$ for complex conjugation.
We use $I$ and $\nullmat$ to denote the identity and null matrices with appropriate sizes, respectively.
The eigenvalues of $A$ are doubly degenerated, or doubly multiple, with the $J$-symmetry. 
Let $x$ be an eigenvector of $A$ associated 
to the eigenvalue $\lambda$, satisfying $A x = \lambda x$. 
The vector defined by
\begin{align}
y \equiv J x^{*},
\label{eq:pairev}
\end{align}
is also the eigenvector of $A$ associated to $\lambda$ 
because of the symmetry~\eqref{eq:symmetry}.

Any Hermitian $J$-symmetric matrix $A$ has the following block structure.
Without loss of generality, $J$ is
\begin{align}
  J &=
  \begin{bmatrix*}[r]
      O  & -I\\
      I  &  O
  \end{bmatrix*},
\end{align}
where the size of each block is $(n/2)\times(n/2)$ and $n$ is an even number.
Based on this, the symmetry \eqref{eq:symmetry} requires
\begin{align}
  A &=
      \begin{bmatrix*}[r]
         A_{11}      & A_{12} \\
         A_{12}^{\HC}& A_{11}^{\TP}
      \end{bmatrix*},
\label{eq:blockform}
\end{align}
where $A_{11}$ and $A_{12}$ are $(n/2)\times(n/2)$ matrices satisfying
$A_{11}^{\HC}=A_{11}$ and $A_{12}^{\TP}=-A_{12}$. 
We can also explicitly construct eigenvectors $x$ and $y=J x^*$ associated with 
an eigenvalue $\lambda$ in the block structure. 
Then $A$ can be diagonalized as
\begin{align}
  A &= U
      \begin{bmatrix*}[r]
          \Lambda & O \\ O & \Lambda
      \end{bmatrix*} U^{\HC},\quad
  U =
      \begin{bmatrix*}[r]
           X_1 &    -X_2^{*}\\
           X_2 & \pp X_1^{*}
      \end{bmatrix*},
\label{eq:diagblockform}
\end{align}
where $\Lambda =\mathrm{diag}(\lambda_1,\lambda_2,\dots,\lambda_{n/2})$, and
$X_1$ and $X_2$ are $(n/2)\times(n/2)$ matrices satisfying 
$X_1^{\HC}X_1+X_2^{\HC}X_2=I$ and
$X_1^{\TP}X_2-X_2^{\TP}X_1=O$.

To investigate the spectrum of a sparse matrix in a large dimension, iterative eigensolver algorithms are generally employed.
Iterative eigensolver algorithms in the literature, such as the Lanczos algorithm for Hermitian matrices, 
do not consider the $J$-symmetry. Therefore, even after convergence on one of a pair of the degenerate eigenvector pairs, 
the algorithm continues to search for the other eigenvector associated to the converged eigenvalue 
even though the other eigenvector can be easily reconstructed from the converged one.
In this article, we restrict ourselves to improving the Lanczos algorithm 
for large Hermitian $J$-symmetric matrices by incorporating the $J$-symmetry property.

We could improve the convergence of the Lanczos algorithm by restricting the search space of 
the Krylov subspace of $A$ to that spanned by one of each pair of the doubly degenerate eigenvector pairs
by imposing complete orthogonality to the subspace spanned by the other eigenvectors paired to them.
However, this strategy cannot be directly realized before knowing the invariant subspace of $A$.
We find that the Krylov subspace $\mathcal{K}_k(A,v)=\mathrm{span}\{v,Av,\dots,A^{k-1}v\}$ generated 
with the Lanczos algorithm with a starting vector $v$ 
is orthogonal to $\mathcal{K}_k(A,Jv^*)=\mathrm{span}\{Jv^*,A Jv^*,\dots,\allowbreak A^{k-1} Jv^*\}$ 
with the same Lanczos algorithm 
with the starting vector $Jv^*$.
On the basis of this property, we can obtain only one half of the degenerate eigenvectors in $\mathcal{K}_k(A,v)$ and
can reconstruct the other half via~\eqref{eq:pairev}.
This can achieve the strategy just stated above.

Early attempts have been made to incorporate the symmetry property or structure of matrices into eigensolver algorithms 
in~\cite{DONGARRA198427,PETKOVIVANOV1994} in which
dense matrix algorithms were developed for the Hermitian $J$-symmetric matrices that appeared 
in a quantum mechanical system with time-reversal and inversion symmetry.
The dense matrix algorithms retain or respect the matrix structure \eqref{eq:blockform} 
during the computation~\cite{DONGARRA198427,PETKOVIVANOV1994}.
To focus on the algorithm for large sparse matrices, 
we do not discuss algorithms for dense matrices in this paper.
The one similar to our algorithm that incorporates the symmetry properties of matrices
was studied in~\cite{BENNER199775,BENNER2011578,BENNER2018407,doi:10.1137/S1064827500366434}, 
in which the Lanczos type eigensolvers for complex $J$-skew-symmetric matrices or 
Hamiltonian matrices were investigated.\footnote{
The definition of the $J$-symmetry employed in~\cite{BENNER199775,BENNER2011578,BENNER2018407,doi:10.1137/S1064827500366434}
is opposite to that of~\cite{DONGARRA198427,PETKOVIVANOV1994} and ours; therefore, their matrices are $J$-skew-symmetric compared to our definition.}
Because the symmetry treated in~\cite{BENNER199775,BENNER2011578,BENNER2018407,doi:10.1137/S1064827500366434}
is different from the Hermitian $J$-symmetry,
their algorithm cannot be directly applied to Hermitian $J$-symmetric matrices, even though it is based on the Lanczos algorithm.
See also~\cite{doi:10.1007/3-540-28502-4} for numerical algorithms for structured eigenvalue problems.

This paper is organized as follows. In the next section, we prove
the orthogonality between the Krylov subspace 
$\mathcal{K}_k(A,v)$ generated with the Lanczos iteration to $\mathcal{K}_k(A,Jv^*)$ with the same Lanczos iteration. 
Having observed the orthogonality, we describe the restarting method based on 
the Krylov--Schur transformation method~\cite{Stewart:2001:KAL:587707.587809} and
the thick-restart method~\cite{WU1999156,doi:10.1137/S0895479898334605}.
Then, we propose the thick-restart Lanczos algorithm for Hermitian $J$-symmetric matrices in Section~\ref{sec:3}.
We estimate the computational cost in terms of the matrix-vector multiplication.
In Section~\ref{sec:4}, we test the proposed algorithm for two types of the Hermitian $J$-symmetric matrix.
The one type is an artificial randomly generated matrix satisfying the structure of \eqref{eq:blockform} and \eqref{eq:diagblockform}, and
the other matrix originates from quantum field theory. 
The convergence behavior and the computational cost are compared with those of 
the standard thick-restart Lanczos algorithm.
We summarize this paper in the last section.

\section{Lanczos iteration and \texorpdfstring{$J$}{J}-symmetry}
\label{sec:2}
The Lanczos iteration transforms $A$ to a tridiagonal form by generating the orthonormal basis vectors. 
The Lanczos decomposition after $m$-step iteration starting with a unit vector $v_1$ is given by
\begin{align}
A V_{m} = V_{m} T_{m} + \beta_{m} v_{m+1} e_{m}^{\TP},
\label{eq:LancV}
\end{align}
where $V_m=\left[v_1,v_2,\dots,v_m\right]$ and $v_j \in \mathbb{C}^n$, 
and $e_m$ is the $m$-dimensional unit vector in the $m$-th direction.
The basis vectors are orthonormal, $V_{m+1}^{\HC}V_{m+1}=I$.
$T_m$ denotes the $m\times m$ tridiagonal matrix that is given by
\begin{align}
T_m =
  \begin{bmatrix*}[c]
     \alpha_1 & \beta_1  &  \\
     \beta_1  & \alpha_2 & \beta_2 & \\
              & \beta_2  & \alpha_3 & \beta_3 \\
              &          & \ddots   & \ddots & \ddots \\
              &          &          & \beta_{m-2} & \alpha_{m-1} & \beta_{m-1} \\
              &          &          &             &  \beta_{m-1} & \alpha_m 
  \end{bmatrix*}.
\end{align}
The approximate eigenpairs are obtained from the eigenpairs of $T_m$. 
Because of the Hermiticity property of $A$ and the recurrence structure of the Lanczos iteration, 
all $\alpha_j$ and $\beta_j$ can be taken to be real.
Thus, $T_m$ becomes a real symmetric matrix. 
In the standard Lanczos iteration, the Hermitian symmetry of $A$ is respected in the form $T_m$.

We also investigate the $J$-symmetry property of decomposition~\eqref{eq:LancV}. 
After taking the complex conjugate of~\eqref{eq:LancV} followed by multiplying it by $J$ from the left-hand side,
we have
\begin{align}
JA^{*} V_m^* =  J V_m^* T_m + \beta_{m} Jv_{m+1}^* e_m^{\TP}.
\end{align}
Using the $J$-symmetry and Hermiticity, $JA^*=A^{\HC}J=AJ$, and by defining $w_j \equiv Jv_j^*$, we obtain
\begin{align}
A W_m &=  W_m T_m + \beta_{m} w_{m+1} e_m^{\TP},
\label{eq:LancW}
\\
W_m &= \left[ w_1, w_2, \dots, w_m\right].
\end{align}
The columns of $W_{m+1}$ are also orthonormal, $W_{m+1}^{\HC}W_{m+1}=I$. 
Consequently, the vectors $W_{m+1}$ have the same Lanczos decomposition as that of $V_{m+1}$ when 
it starts from $w_1 = Jv_1^*$.
Furthermore, we can prove the orthogonality between $W_{m+1}$ and $V_{m+1}$.

To implement the thick-restart method~\cite{WU1999156,doi:10.1137/S0895479898334605}, 
we show the $J$-symmetry property for a generalized decomposition similar to 
the Krylov--Schur decomposition~\cite{Stewart:2001:KAL:587707.587809}
instead of the Lanczos decomposition in the following part of the paper. 
After preparing two lemmas related to the $J$-symmetry, we prove 
the main theorem for the orthogonality property between $W_{m+1}$ and $V_{m+1}$
on the generalized decomposition.
Subsequently, the orthogonality properties for the Lanczos decomposition and the thick-restart method
are proved as corollaries.

We first show orthogonal properties between $v$ and $w=Jv^*$.
\begin{lemma}
\label{lemma1}
Let $v$ be an arbitrary vector in $\mathbb{C}^n$,
and $J$ and $A$ be matrices satisfying~\eqref{eq:symmetry}.
Then, the vector $w$ defined by $w \equiv J v^*$ satisfies
\begin{align}
w^{\HC} v   & = 0,\\
w^{\HC} A v & = 0.
\end{align}
\end{lemma}
\begin{proof}
From the definition of $w$, it follows that
\begin{align}
  w^{\HC} v = (J v^*)^{\HC} v = v^{\TP} J^{\HC} v = v^{\TP} J^{\TP} v = - v^{\TP} J v,
\label{eq:wTvidentity}
\end{align}
where  $J^{\TP} = -J$ is used. 
The identity 
$v^{\TP} Jv = v^{\TP} J^{\TP} v$ and
\eqref{eq:wTvidentity} yield
\begin{align}
  w^{\HC} v = -w^{\HC} v.
\end{align}
Thus, $w^{\HC} v = 0$. Similarly, 
\begin{align}
  w^{\HC} A v = (J v^*)^{\HC} A v = v^{\TP} J^{\HC} A v = v^{\TP} J^{\TP} A v = -v^{\TP} J A v = - v^{\TP} A^{\TP} J v,
\label{eq:wTAvidentity}
\end{align}
where $J^{\TP}=-J$ and $JA=A^{\TP}J$ are used.
The identity $v^{\TP} A^{\TP} J v= v^{\TP} J^{\TP} A v$
and \eqref{eq:wTAvidentity} yield
\begin{align}
  w^{\HC} A v =  - w^{\HC} A v.
\end{align}
Thus, $w^{\HC} A v = 0$.
\qed
\end{proof}

We have the following lemma that is a generalization of the relation between \eqref{eq:LancV} and~\eqref{eq:LancW}.
\begin{lemma}
 \label{lemma2}
 Let $v_1,v_2,\dots,v_k,v_{k+1}$ be the vectors having the following relation:
 \begin{align}
   A V_k = V_k S_k + v_{k+1} b^{\TP},
\label{eq:KSV}
 \end{align}
 where $V_k=[v_1,\dots,v_k]$, $S_k$ is a matrix in $\mathbb{R}^{k\times k}$,
 and $b$ is a vector in $\mathbb{R}^k$
 for a Hermitian $J$-symmetric matrix $A$ satisfying~\eqref{eq:symmetry}.
 Then, the vectors $w_j = J v_j^*$ $(j=1,\dots,k+1)$ satisfy the following decomposition:
 \begin{align}
  A W_k = W_k S_k + w_{k+1} b^{\TP},
\label{eq:KSW}
 \end{align}
where $W_k=JV_k^*=[w_1,\dots,w_k]$.
\end{lemma}
\begin{proof}
By taking the complex conjugate of~\eqref{eq:KSV} followed by multiplying it by $J$ from the left-hand side,
 we obtain~\eqref{eq:KSW} using $JA^{\TP}=AJ$ and $A^* = A^{\TP}$.
\qed
\end{proof}

Using Lemmas~\ref{lemma1} and~\ref{lemma2}, 
we can show the orthogonality properties between $W_{k+1}$ and $V_{k+1}$ as follows.
\begin{theorem}
\label{theo:j1}
Let $V_{k+1}=[v_1,\dots,v_{k+1}]$ be a matrix satisfying~\eqref{eq:KSV}, and $W_{k+1}=JV_{k+1}^*$.
If the matrices $V_{k}$ and $W_{k}$ are orthogonal and $A$-orthogonal to each other:
$V_k^{\HC}W_k=\nullmat$ and $V_k^{\HC}AW_k=\nullmat$,
then the matrices $V_{k+1}$ and $W_{k+1}$ also satisfy the following orthogonality relations:
\begin{align}
 V_{k+1}^{\HC}   W_{k+1} & = \nullmat,
\label{eq:VHW}
\\
 V_{k+1}^{\HC} A W_{k+1} & = \nullmat.
\label{eq:VHAW}
\end{align}
\end{theorem}
\begin{proof}
The decomposition \eqref{eq:KSW} follows from Lemma~\ref{lemma2}.
Multiplying $W_k^{\HC}$ and $V_k^{\HC}$ to~\eqref{eq:KSV} and~\eqref{eq:KSW}, respectively, yields
\begin{align}
W_k^{\HC} A V_k &=  W_k^{\HC} V_k S_k + W_k^{\HC} v_{k+1} b^{\TP},\\
V_k^{\HC} A W_k &=  V_k^{\HC} W_k S_k + V_k^{\HC} w_{k+1} b^{\TP}.
\end{align} 
From the premise that $W_k^{\HC} A V_k=V_k^{\HC} A W_k=\nullmat$ and $W_k^{\HC} V_k =V_k^{\HC} W_k=\nullmat$, 
it follows that
\begin{align}
W_k^{\HC} v_{k+1} = \nullmat,
\label{eq:Wkvk1}
\\
V_k^{\HC} w_{k+1} = \nullmat.
\label{eq:Vkwk1}
\end{align} 
Together with Lemma~\ref{lemma1} and the premise, we find 
\begin{align}
  V_{k+1}^{\HC}W_{k+1}& = \nullmat.
\end{align}

Multiplying $w_{k+1}^{\HC}$ and $v_{k+1}^{\HC}$ to~\eqref{eq:KSV} and~\eqref{eq:KSW}, respectively, yields
\begin{align}
w_{k+1}^{\HC} A V_k &=  w_{k+1}^{\HC} V_k S_k + w_{k+1}^{\HC} v_{k+1} b^{\TP},
\label{eq:wk1AVk}
\\
v_{k+1}^{\HC} A W_k &=  v_{k+1}^{\HC} W_k S_k + v_{k+1}^{\HC} w_{k+1} b^{\TP}.
\label{eq:vk1AWk}
\end{align} 
Because of~\eqref{eq:Wkvk1} and~\eqref{eq:Vkwk1} as well as Lemma~\ref{lemma1},
the right-hand sides of~\eqref{eq:wk1AVk} and~\eqref{eq:vk1AWk} vanish.
Thus,
\begin{align}
w_{k+1}^{\HC} A V_k & = \nullmat,\\
v_{k+1}^{\HC} A W_k & = \nullmat.
\end{align} 
Together with Lemma~\ref{lemma1} and the premise, we find 
\begin{align}
  V_{k+1}^{\HC} A W_{k+1}& = \nullmat.
\end{align}
Therefore, $V_{k+1}$ and $W_{k+1}$ are orthogonal and $A$-orthogonal to each other.
\qed
\end{proof}

Using Theorem~\ref{theo:j1} as well as Lemmas~\ref{lemma1} and~\ref{lemma2},
we can show that the Lanczos vectors $V_{m+1}$ generated with~\eqref{eq:LancV}
are orthogonal and $A$-orthogonal to $W_{m+1}=JV_{m+1}^*$.
\begin{corollary}
\label{coro1}
    The $m$-step Lanczos vectors $V_{m+1}$ with $m\ge 1$ 
    generated with a unit vector $v_1$ for a Hermitian $J$-symmetric matrix $A$ satisfy
    \begin{align}
      V_{m+1}^{\HC} W_{m+1} = \nullmat,\quad V_{m+1}^{\HC} A W_{m+1} = \nullmat,
    \end{align}
with
\begin{align}
  W_{m+1} = J V_{m+1}^*,
\end{align}
when no breakdown occurs.
\end{corollary}
\begin{proof}
Because the Lanczos decomposition \eqref{eq:LancV} is a particular form of~\eqref{eq:KSV} 
with $k\to m$, a real matrix $S_m \to T_m$ and a real vector $b \to \beta_m e_m$, 
Lemma~\ref{lemma2} can be applied to obtain \eqref{eq:LancW}.
Because $w_1^{\HC} v_1=0$ and $w_1^{\HC} A v_1=0$ hold from Lemma~\ref{lemma1},
we can apply Theorem~\ref{theo:j1} to \eqref{eq:LancV} and \eqref{eq:LancW} when $m=1$.
Moreover, the Lanczos decomposition retains its form applicable to Theorem~\ref{theo:j1} for any $m>1$.
Therefore, the corollary follows from Theorem~\ref{theo:j1} and Lemmas~\ref{lemma1} and \ref{lemma2} by induction. 
\qed
\end{proof}

We cannot simultaneously find degenerate pairs of eigenvectors 
with the standard single-vector Lanczos process.
This is true for computation with exact arithmetic.
However, with finite precision arithmetic, 
the single-vector Lanczos process would generate a small overlap to $W_k$ via round-off errors, 
so that even after the convergence of an eigenvector, a late convergence to the other paired eigenvector 
would be possible. 
Because this behavior is rather accidental, a block-type Lanczos algorithm has to be applied 
to accelerate the convergence for degenerate 
eigenvectors~\cite{doi:10.1137/S1064827501397949,doi:10.1137/1.9780898719581,GOLUB1977361,SHIMIZU2019372,Zhou2008}.

We further investigate the structure of the Lanczos decomposition.
According to Corollary~\ref{coro1} and the Lanczos decompositions
\eqref{eq:LancV} and \eqref{eq:LancW}, the block-type decomposition can be constructed as:
\begin{align}
  A V_{[m]} = V_{[m]} T_{[m]} + \mathcal{V}_{m+1} \mathcal{B}_m E_{[m]}^{\TP},
\end{align}
where we define
\begin{align}
  V_{[m]} &\equiv 
  \begin{bmatrix*}[c]
  \mathcal{V}_1 & \mathcal{V}_2 & \dots & \mathcal{V}_{m-1} & \mathcal{V}_m 
  \end{bmatrix*},
\quad
T_{[m]} \equiv
  \begin{bmatrix*}[l]
    \mathcal{A}_1 \pp& \mathcal{B}_1\pp &  \\
    \mathcal{B}_1 \pp& \mathcal{A}_2\pp & \mathcal{B}_2\pp & \\
                  & \ddots        & \ddots            & \ddots \\
                  &               & \mathcal{B}_{m-2} & \mathcal{A}_{m-1} & \mathcal{B}_{m-1} \\
                  &               &                   & \mathcal{B}_{m-1} & \mathcal{A}_m 
  \end{bmatrix*},
\notag\\
\mathcal{V}_j &\equiv
                \begin{bmatrix*}[c]
                    v_j, w_j
                \end{bmatrix*},
\quad
\mathcal{A}_j \equiv \mathrm{diag}(\alpha_j,\alpha_j),\quad
\mathcal{B}_j \equiv \mathrm{diag}(\beta_j ,\beta_j ),\quad
E^{\TP}_{[m]} \equiv 
  \begin{bmatrix*}[c]
      O & O & \dots & O & I
  \end{bmatrix*},
\end{align}
where the size of $E^{\TP}_{[m]}$ is $2\times 2m$.
When $m = n/2$, the decomposition should terminate, 
because $V_{[n/2]}$ completely block-tridiagonalize $A$ and
the Krylov subspaces
$\mathcal{K}_{n/2}(A,v)$ and $\mathcal{K}_{n/2}(A,Jv^*)$ span the entire eigenspace of $A$.
Because $\mathcal{K}_{n/2}(A,v)$ and $\mathcal{K}_{n/2}(A,Jv^*)$ are orthonormal and 
\eqref{eq:LancV} and \eqref{eq:LancW} are independent iterations,
the Lanczos iteration for \eqref{eq:LancV} terminates at $m=n/2$ regardless of the iteration for \eqref{eq:LancW}.
We, therefore, can construct eigenvectors from $\mathcal{K}_{n/2}(A,v)$ without the eigenvalue multiplicity associated with the $J$-symmetry.
In other words, the standard Lanczos iteration with exact precision arithmetic 
is enough to find all the eigenvectors without multiplicity.
However, this is impractical because it requires exact precision arithmetic.
With the finite precision, the orthogonality to $\mathcal{K}_{n/2}(A,Jv^*)$ is not maintained because of round-off errors
and, eventually, eigenvectors, including multiplicity, could be extracted from 
the single-vector Lanczos iteration, as stated previously.

By using Corollary~\ref{coro1} and with the above analysis, we can construct a Lanczos type algorithm in which
the orthogonality to $W_k$ is enforced to search for eigenvectors without 
multiplicity of eigenvalues associated to the $J$-symmetry.
Additionally, the other vectors paired to them can be easily reconstructed. 
However, for a practical numerical algorithm of the Lanczos type iteration, the iteration should 
terminate at a finite step, and a restarting mechanism is required~\cite{etna_vol2_pp1-21,WU1999156,doi:10.1137/S0895479898334605}.
The most useful and simplest but effective restarting method is 
the so-called thick-restart method~\cite{WU1999156,doi:10.1137/S0895479898334605} that
is a specialization of the Krylov--Schur transformation~\cite{Stewart:2001:KAL:587707.587809}
to Hermitian matrices. To involve the thick-restart method to 
the Lanczos algorithm with the $J$-symmetry, we have to prove the orthogonality between $V_k$ and $W_k$ after 
the Krylov--Schur transformation and restarting.
To achieve this, we have the following corollary.

\begin{corollary}
\label{coro2}
Let $V_{m+1}$ and $W_{m+1}$ be the orthonormal matrices containing basis vectors generated with the $m$-step Lanczos process,
\eqref{eq:LancV} and \eqref{eq:LancW},  and
$Z_k$ be an orthonormal matrix in $\mathbb{R}^{k\times k}$.
The Krylov--Schur transformation with $Z_k$ on the Lanczos decomposition
is defined by
\begin{align}
A U_m &=  U_m (Z_m^{-1} T_m Z_m) +  v_{m+1} b_m^{\TP},
\label{eq:LancVkZk} 
\\
A Q_m &=  Q_m (Z_m^{-1} T_m Z_m) +  w_{m+1} b_m^{\TP},
\label{eq:LancWkZk}
\\
U_m       &\equiv V_m Z_m,\\
Q_m       &\equiv W_m Z_m,\\
b_m^{\TP} &\equiv \beta_{m} e_m^{\TP} Z_m.
\end{align} 
Then, the matrices $U_{m+1} = [ U_{m}, v_{m+1} ]$ and  $Q_{m+1} =  [Q_{m}, w_{m+1} ]$ satisfy
the orthogonal relations: $U_{m+1}^{\HC} Q_{m+1}=\nullmat$ and $U_{m+1}^{\HC} A Q_{m+1}=\nullmat$.
\end{corollary}
\begin{proof}
Because $U_m$ and $Q_m$ satisfy $U_m^{\HC} Q_m=\nullmat$ and $Q_m^{\HC} A U_m=\nullmat$
and the decompositions \eqref{eq:LancVkZk} and \eqref{eq:LancWkZk} 
are particular forms of the decomposition in Lemma~\ref{lemma2},
the orthogonality and $A$-orthogonality between $U_{m+1}$ and $Q_{m+1}$ simply follow
from Theorem~\ref{theo:j1}.
\qed
\end{proof}

For the thick-restart method, $Z_m$ is chosen to diagonalize $T_m$, and 
the dimension of the decomposition is reduced from $m$ to $k<m$ 
with a selection criterion for vectors $V_k \leftarrow V_m$. 
In the reduction, the last vectors are kept hold as $v_{k+1} \leftarrow v_{m+1}$ and  $w_{k+1} \leftarrow w_{m+1}$
to retain the decomposition form properly.
The orthogonality properties of the new basis $(U_{m+1},Q_{m+1})$ and the reduced basis $(U_{k+1},Q_{k+1})$ 
still hold according to Corollary~\ref{coro2}. 
After restarting, the Lanczos iteration continues to keep the decomposed form applicable to Theorem~\ref{theo:j1}.
Because the structures of the Lanczos and Krylov--Schur decompositions for $W_k$ and $Q_k$
are the same as those for $V_k$ and $U_k$, respectively,
we do not need to explicitly iterate the Lanczos algorithm for $W_k$ and $Q_k$.
Consequently, we can continue the Lanczos thick-restart cycle only in the subspace $\mathcal{K}_k(A,v)$ 
that is orthogonal and $A$-orthogonal to $\mathcal{K}_k(A,Jv^*)$.
We note that according to Corollaries~\ref{coro1} and \ref{coro2},
all the eigenpairs without multiplicity can be obtained with the thick-restart Lanczos algorithm
using exact precision arithmetic.
This is impractical and we must enforce the orthogonality between $\mathcal{K}_k(A,v)$ 
and $\mathcal{K}_k(A,Jv^*)$ for a practical algorithm.

\section{Thick-restart Lanczos algorithm with \texorpdfstring{$J$}{J}-symmetry}
\label{sec:3}

Based on Theorem~\ref{theo:j1} as well as Corollaries~\ref{coro1} and \ref{coro2}, we construct a thick-restart Lanczos 
algorithm for Hermitian $J$-symmetric matrices (TRLAN--JSYM)
which efficiently searches for eigenvectors without the multiplicity of eigenvalues in $\mathcal{K}_k(A,v)$.
Algorithm~\ref{alg:TRLANJSYM} shows the TRLAN--JSYM algorithm.
The Lanczos iteration with the $J$-symmetry is described in Algorithm~\ref{alg:LANCZOSJSYM}. 
We include the invert mode for the small eigenvalues.

The main difference from the standard thick-restart Lanczos algorithm (TRLAN) is
in Algorithm~\ref{alg:LANCZOSJSYM}, where
we simultaneously construct $w_{j}$ using $w_{j}=Jv_j^*$ and
enforce the orthogonality of $V_{j+1}$ to $W_j =J V_j^*$ to avoid the contamination of 
the search space $\mathcal{K}_k(A,v)$ from $\mathcal{K}_k(A,Jv^*)$.

We will compare the efficiency between the TRLAN--JSYM and the standard TRLAN algorithms in Section~\ref{sec:4}.
For the comparison, we estimate the computational cost of the TRLAN--JSYM and the standard TRLAN algorithms as follows.
We count the total number of matrix-vector multiplication $Av_j$ or $A^{-1}v_j$ 
contained in the Lanczos step Algorithm~\ref{alg:LANCZOSJSYM}.
For the invert mode with a large sparse matrix, it could require an iterative linear solver
and a computational cost to obtain $A^{-1}v_j$.
To focus on the computational cost comparison between the TRLAN--JSYM and TRLAN algorithms,
assuming the computation cost of a single inversion is identical between the two algorithms,
we do not count the cost involved in the inversion
and regard the single inversion operation $A^{-1}v_j$ as one matrix-vector multiplication for the invert mode.
We neglect the cost that explicitly computes the true residual at line 28 in Algorithm~\ref{alg:TRLANJSYM}.
For the first outer iteration, the count is $m$, and after restarting, it is $m-k$.
The thickness $k$ for restarting is defined by 
\begin{align}
k=\min(\mathrm{icnv}+\mathrm{mwin},m-1),
\label{eq:dynamicwin}
\end{align}
as shown in line 48 of Algorithm~\ref{alg:TRLANJSYM},
where $\mathrm{icnv}$ is the number of converged eigenvectors and
$\mathrm{mwin}$ is the initial thickness for restarting.
When all desired eigenvectors are obtained at an outer iteration $N_{\mathrm{conv}}$, 
the upper and lower bounds of the total number of matrix-vector multiplication $N_{\mathrm{MV}}$ 
is estimated as
\begin{align}
  m + (m-\mathrm{mwin}-\mathrm{nev})(N_{\mathrm{conv}}-1)
< N_{\mathrm{MV}} 
< m + (m-\mathrm{mwin}             )(N_{\mathrm{conv}}-1),
\label{eq:MVcost}
\end{align}
where $\mathrm{nev}$ is the number of desired eigenvectors without the multiplicity of $J$-symmetry.
The inequality \eqref{eq:MVcost} follows from the fact that $\mathrm{icnv}$ increases monotonically from zero to 
$\mathrm{nev}$ toward the convergence.

The same cost estimate can be derived for the standard TRLAN algorithm. 
The TRLAN algorithm can be obtained by removing $W_m$ from Algorithms~\ref{alg:TRLANJSYM} and \ref{alg:LANCZOSJSYM}.
Therefore, the cost bound for the TRLAN algorithm is identical to \eqref{eq:MVcost}.
However, to find all eigenvectors paired with the $J$-symmetry using the TRLAN algorithm,
$\mathrm{nev}$ for the TRLAN must be double that of the TRLAN--JSYM.
Thus it is natural to double all parameters for the TRLAN algorithm than those of the TRLAN--JSYM.
Therefore, the upper and lower bounds of the total 
number of matrix-vector multiplication $N_{\mathrm{MV}}$ of the TRLAN algorithm is
\begin{align}
  2(m + (m-\mathrm{mwin}-\mathrm{nev})(N'_{\mathrm{conv}}-1))
< N_{\mathrm{MV}} 
< 2(m + (m-\mathrm{mwin}             )(N'_{\mathrm{conv}}-1)),
\label{eq:MVcostN}
\end{align}
where the parameters $(\mathrm{nev},\mathrm{mwin},m)$ are those of the TRLAN--JSYM, and
we introduced $N'_{\mathrm{conv}}$ for the number of outer iterations because 
it could be different from that of the TRLAN--JSYM.

Although the computational cost of the Gram--Schmidt orthonormalization 
of the Lanczos step is minor compared to that of the matrix-vector multiplication,
we briefly discuss the cost here.
As shown in Algorithm~\ref{alg:LANCZOSJSYM} for the TRLAN--JSYM, 
the number of vectors to orthonormalize is $2m$ 
and the cost scales with $O((2m)^2)$.
On the other hand, it scales with $O(m^2)$ to orthonormalize $V_m$ with the standard Lanczos algorithm.
As described above, it is natural to double the parameters for the TRLAN.
The cost of the Lanczos part to orthonormalize $V_{2m}$ then becomes $O((2m)^2)$.
Therefore, the scaling of the cost is the same for both algorithms.

Our naive estimates on the computational cost are \eqref{eq:MVcost} and \eqref{eq:MVcostN},
where we assume that the parameters of the TRLAN is twice as large as those of the TRLAN--JSYM.
Although $N_{\mathrm{conv}}$ and $N'_{\mathrm{conv}}$ depend on 
$(\mathrm{nev},\mathrm{mwin},m)$ and the algorithm itself, 
if $N_{\mathrm{conv}}\simeq N'_{\mathrm{conv}}$ holds, 
the TRLAN--JSYM algorithm has a better performance than the TRLAN algorithm.
We will see whether the condition $N_{\mathrm{conv}}\simeq N'_{\mathrm{conv}}$ holds or not
in the numerical tests on the two types of the matrices in the next section.

So far, we have described the single-vector Lanczos iteration type algorithm to introduce the TRLAN--JSYM.
If the matrix $A$ has a dense cluster of eigenvalues or multiple eigenvalues other than those with the $J$-symmetry, 
we need to incorporate the block type Lanczos iteration in the algorithm for efficiency.
We can extend the proposed algorithm to the blocked version in the same manner as it was conducted
for the standard thick-restart Lanczos algorithm~\cite{SHIMIZU2019372,Zhou2008}.
The study on the block version will be addressed in future studies and we have only shown
the single vector version to demonstrate the idea for simplicity.

\begin{algorithm}[h]
 \caption{The thick-restart Lanczos algorithm for a  Hermitian $J$-symmetric matrix $A$ (TRLAN--JSYM).}
 \label{alg:TRLANJSYM}
\begin{algorithmic}[1]
  \Require{Maximum Krylov subspace dimension size $m$, 
           restart window size $\mathrm{mwin}$, and number of desired eigenpairs $\mathrm{nev}$. }
  \Ensure{Eigenpairs $(x_i, \lambda_i)$ and residual norms $||r_i||=||A x_i-\lambda_i x_i||$ 
           for $i=1,\dots,\mathrm{nev}$ in $V(:,1:\mathrm{nev}), \mathrm{ev}(1:\mathrm{nev}), \mathrm{res}(1:\mathrm{nev})$.}
  \State{$k=0$}
  \State{$v_1 = 1$;  $v_1 = v_1/||v_1||$ \Comment{Initial unit vector}}
  \State{$w_1 = Jv_1^*$                  \Comment{Initial dual vector}}
  \Loop
  \State{\Call{LANCZOS\_JSYM}{{$m,k,V_{m+1},W_{m+1},\bar{T}_{m+1}$}} \Comment{$m-k$-step Lanczos with $J$-symmetry}}
  \State{$T_m = Z_m \Lambda_m Z_m^{\TP}$       \Comment{Compute eigenpairs of $T_m$}}
  \State{Move desired eigenpairs in the top dimensions of $Z_m$ and $\Lambda_m$ by sorting.}
  \State{$V_m := V_m Z_m$                      \Comment{Compute approximate eigenvectors}}
  \State{$T_m=0$\Comment{Compute the Krylov--Schur transformation for $\bar{T}_{m+1}$}}
  \For{$i=1,\dots,m$}
    \State{$t_{i,i} = \lambda_i$}
    \State{$t_{m+1,i}=t_{m+1,m}z_{m,i}$}
  \EndFor
  \If{Normal Mode} \Comment{Compute estimated residuals}
    \For{$i=1,\dots,\mathrm{nev}$} 
      \State{$\mathrm{ev}(i)=\lambda_i$}
      \State{$\mathrm{res\_est}(i)=|t_{m+1,i}|$} 
    \EndFor
  \ElsIf{Invert Mode} \Comment{Compute estimated residuals}
    \State{$c = ||A v_{m+1}||$}
    \For{$i=1,\dots,\mathrm{nev}$} 
      \State{$\mathrm{ev}(i)=1/\lambda_i$}
      \State{$\mathrm{res\_est}(i)= c |t_{m+1,i} \mathrm{ev}(i)|$}
    \EndFor
  \EndIf
  \For{$i=1,\dots,\mathrm{nev}$} 
    \If{$\mathrm{res\_est}(i) < \mathrm{tol}$} \Comment{Check true residuals}
      \State{$\mathrm{res}(i) = ||A v_i - v_i \mathrm{ev}(i)||$}
      \If{$\mathrm{res}(i) < \mathrm{tol}$}
         \State{$\mathrm{is\_convd}(i)=.\mathrm{TRUE}.$}
      \Else
         \State{$\mathrm{is\_convd}(i)=.\mathrm{FALSE}.$}
      \EndIf
    \EndIf
  \EndFor
  \State{Move converged eigenpairs in the top dimensions of $Z,T,V,\mathrm{res},
         \mathrm{res\_est},\mathrm{ev}$ with the key $\mathrm{is\_convd}$ by sorting.}
  \State{$\mathrm{icnv}=0$}
  \For{$i=1,\dots,\mathrm{nev}$}
    \If{$\mathrm{is\_convd}(i)==.\mathrm{TRUE}.$}
      \State{$t_{m+1,i}=0$                    \Comment{Decouple converged subspace}}
      \State{$\mathrm{icnv}=\mathrm{icnv}+1$  \Comment{Count number of converged eigenpairs}}
    \EndIf
  \EndFor
  \If{$\mathrm{icnv}==\mathrm{nev}$}
    \State{\bf{Exit Loop}}
  \EndIf
  \State{\Comment{Shrink Krylov--Schur decomposition to $k+1$ dimension}}
  \State{$k=\mathrm{MIN}(\mathrm{icnv}+\mathrm{mwin},m-1)$}
  \For{$i=1,\dots,k$}
    \State{$t_{k+1,i}=t_{m+1,i}$}
  \EndFor
  \State{$v_{k+1} = v_{m+1}$}
  \State{$W_{k+1} = J V_{k+1}^*$}
  \EndLoop
\end{algorithmic}
\end{algorithm}

\begin{algorithm}[h]
\caption{$m-k$-step Lanczos iteration for a  Hermitian $J$-symmetric matrix $A$.}
\label{alg:LANCZOSJSYM}
\begin{algorithmic}[1]
\Procedure{LANCZOS\_JSYM}{$n,m,k,V_{m+1},W_{m+1},\bar{T}_m$}
\Require{$V_{k+1}, W_{k+1}, \bar{T}_k$}
\Ensure{$V_{m+1}, W_{m+1}, \bar{T}_m$}
\State{$\gamma=\sqrt{2}$ \Comment{Reorthogonalization threshold parameter}}
\For{$j=k+1,\dots,m$}
  \If{Normal Mode}
     \State{$v_{j+1}= A v_j$}
  \ElsIf{Invert Mode}
     \State{$v_{j+1} = A^{-1} v_j$}
  \EndIf
  \State{$t_{j,j} = v_j^{\HC} v_{j+1}$}
  \State{$v_{j+1} := v_{j+1} - v_j t_{j,j}$}
  \State{$b_0 = ||v_{j+1}||$}
  \Loop
    \For{$i=1,\dots,j$} 
      \State{$c = w_i^{\HC} v_{j+1}$}
      \State{$v_{j+1} := v_{j+1} - w_i  c$ \Comment{Orthogonalization to $W_j$}}
      \State{$c = v_i^{\HC} v_{j+1}$}
      \State{$v_{j+1} := v_{j+1} - v_i  c$}
    \EndFor
    \State{$b_1 = ||v_{j+1}||$}
    \If{ $b_1 \gamma > b_0$}
      \State{\textbf{Exit Loop}}
    \EndIf
    \State{$b_0=b_1$}
  \EndLoop
  \State{$v_{j+1} := v_{j+1}/b_0$}
  \State{$t_{j+1,j} = b_0$}
  \State{$w_{j+1} = J v_{j+1}^*$  \Comment{Construct $W_{j+1}$}}
\EndFor
\EndProcedure
\end{algorithmic}
\end{algorithm}

\section{Numerical Test}
\label{sec:4}

In this section, we show two numerical tests to explore the efficiency of 
the TRLAN--JSYM algorithm compared to the TRLAN algorithm for Hermitian $J$-symmetric matrices.
The first test is conducted for randomly generated Hermitian $J$-symmetric matrices satisfying the structure 
\eqref{eq:blockform}.
The second test is applied to a matrix in quantum field theory.
We refer to these two test cases as Case A and Case B, respectively.

We implement both the algorithms, TRLAN--JSYM and TRLAN, with Fortran 2003. 
The numerical tests were performed on a single node of the subsystem A of 
the ITO supercomputer system of Kyushu university~\cite{ITOSYSTEM}.
The code is parallelized using OpenMP and the Intel MKL library and executed with 36 threads.

\subsection{Case A}
\subsubsection{Definition of the Test Matrix (Case A)}

\begin{figure}[t]
  \centering
  \includegraphics[clip,trim=0 0 0 0,scale=\figscale]{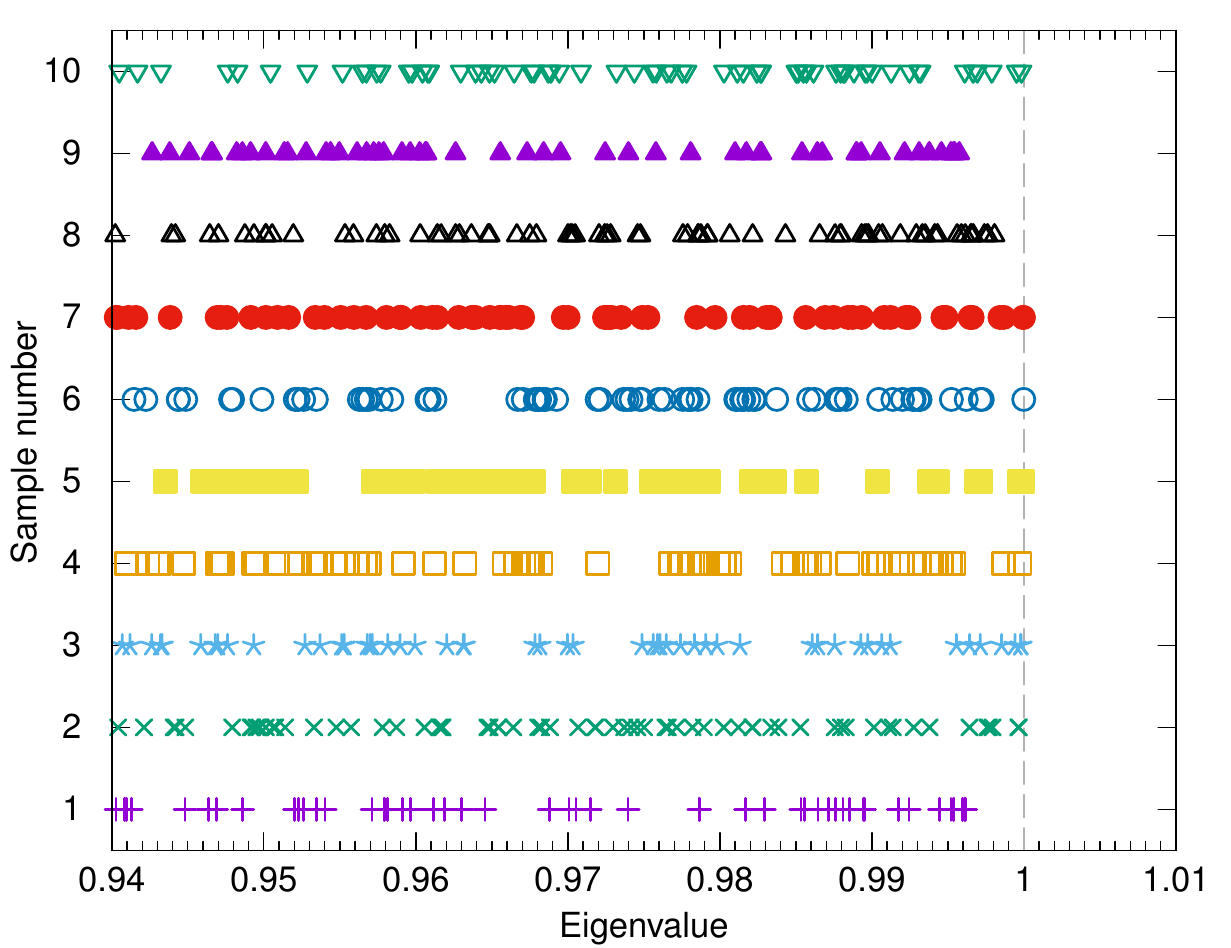}
  \caption{Large eigenvalue distribution of the ten random matrices}
  \label{fig:randmatev}
\end{figure}

We generated ten matrices with a size of $2000\times 2000$. 
These sample matrices are randomly generated, as explained below.
Although these matrices are dense, we employ them 
to explore the proposed algorithm~\ref{alg:TRLANJSYM}.

To randomly generate matrices $A$ with the structure \eqref{eq:blockform},
we employ \eqref{eq:diagblockform}.
The eigenvalues $\Lambda$ are generated from uniformly distributed random real numbers in $\{x\in\mathbb{R}: 0<x<1\}$,
and the elements of matrices $X_1$ and $X_2$ are constructed
from uniformly distributed random complex numbers in $\{z \in \mathbb{C}: -1<\mathrm{Re}(z)< 1, -1 < \mathrm{Im}(z) <1\}$. 
The constraints 
$X_1^{\HC}X_1+X_2^{\HC}X_2=I$ and $X_1^{\TP}X_2-X_2^{\TP}X_1=O$
are then imposed by a Gram--Schmidt algorithm similar to that used in 
Algorithm~\ref{alg:LANCZOSJSYM}. 
Fig.~\ref{fig:randmatev} shows large eigenvalues for the ten random matrices.
We solve the largest several eigenvalues with the TRLAN--JSYM and TRLAN algorithms 
and compare the convergence behavior.

\begin{table}[t]
\centering
    \caption{Algorithm parameters and statistics for convergence (Case A)}
    \label{tab:algparmsA}
    \begin{tabular}{lS[table-format=2.0,table-column-width=2em]
                     S[table-format=2.0,table-column-width=2em]
                     S[table-format=2.0,table-column-width=2em]
                     c
                     c}
\hline\noalign{\smallskip}
      & {$\mathrm{nev}$} & 
        {$\mathrm{mwin}$} &
        {$m$} &
        {$N_{\mathrm{conv}}$}&
        {$N_{\mathrm{MV}}$}
\\
\noalign{\smallskip}\hline\noalign{\smallskip}
TRLAN--JSYM &  5 &  10 &  50 & [  6,  9.1, 14] & [ 248, 369.6, 559]\\
            &    &     & 100 & [  3,  4.1,  6] & [ 280, 377.5, 546]\\
            &    &     & 150 & [  2,  2.8,  4] & [ 290, 400.8, 567]\\
            &    &     & 200 & [  2,  2.1,  3] & [ 389, 408.7, 578]\\
            & 10 &  20 &  50 & [ 10, 12.6, 16] & [ 306, 378.2, 484]\\
            &    &     & 100 & [  4,  4.9,  6] & [ 332, 404.6, 497]\\
            &    &     & 150 & [  3,  3.2,  4] & [ 401, 432.0, 537]\\
            &    &     & 200 & [  2,  2.3,  3] & [ 379, 432.1, 557]\\
\noalign{\smallskip}\hline\noalign{\smallskip}
TRLAN       & 10 &  20 & 100 & [  6,  8.3, 12] & [ 480, 655.6, 935]\\
            &    &     & 200 & [  3,  3.8,  5] & [ 555, 693.7, 904]\\
            &    &     & 300 & [  2,  2.6,  3] & [ 572, 740.3, 856]\\
            &    &     & 400 & [  2,  2.1,  3] & [ 773, 813.1,1150]\\
            & 20 &  40 & 100 & [ 10, 12.6, 15] & [ 569, 699.2, 844]\\
            &    &     & 200 & [  4,  4.6,  6] & [ 652, 746.7, 958]\\
            &    &     & 300 & [  3,  3.1,  4] & [ 794, 828.4,1051]\\
            &    &     & 400 & [  2,  2.1,  3] & [ 748, 787.7,1110]\\
\noalign{\smallskip}\hline
    \end{tabular}
\end{table}

\subsubsection{Numerical Results (Case A)}

Table~\ref{tab:algparmsA} shows the algorithmic parameters used in this test.
For the stopping condition of the algorithms, we employ $\mathrm{tol}=10^{-13}$ for tolerance.
We also tabulate 
the number of outer iteration counts $N_{\mathrm{conv}}$ and matrix-vector multiplications $N_{\mathrm{MV}}$
for convergence in the table. 
The minimal, average, and maximal values from the ten samples are shown in square brackets, respectively.

Figs.~\ref{fig:CAs1AevconvALL} and \ref{fig:CAs1AerrALL} are the convergence 
histories of the eigenvalues and corresponding residuals
for the 1st random matrix, respectively.
The left panels show the result with the TRLAN--JSYM algorithm 
with $(\mathrm{nev},\mathrm{mwin},m)=(10,20,50)$,
and the right panels show the result with the TRLAN algorithm with $(20,40,100)$.
We employ this doubled parameter for the TRLAN, as discussed in the previous section.
The behavior of the TRLAN--JSYM is smooth, while it reorders several times for the TRLAN, 
even though the TRLAN algorithm successfully captures all the eigenvalues with multiplicity two.
The sorting algorithm for the eigenvalues 
and the property of the Lanczos iteration to the $J$-symmetry
cause the reordering of the eigenvalues for the TRLAN.
As mentioned in Section~\ref{sec:2},
the TRLAN tends to evaluate one eigenvector of a pair of the doubly degenerate eigenvalues 
in the early stage of the iterations.
Because of the finite precision arithmetic, it loses complete orthogonality to the other half of the degenerate eigenspace during the iterations,
so that the other eigenvalue paired to the converged eigenvalue emerges in the later stage.
Similar behaviors are also observed for other random matrices at the same algorithmic parameter.
For the cases with a larger $m$,
we do not observe the eigenvalue reordering with the TRLAN algorithm, because they quickly converge.
Increasing $m$, $N_{\mathrm{conv}}$ rapidly decreases, as shown in Table~\ref{tab:algparmsA}.
However, $N_{\mathrm{MV}}$ is almost constant or slightly increasing.

\begin{figure}[t]
    \centering
    \begin{minipage}[b]{\figwidth\linewidth}
    \centering
    \includegraphics[clip,trim=0 0 0 0,scale=\figscale]{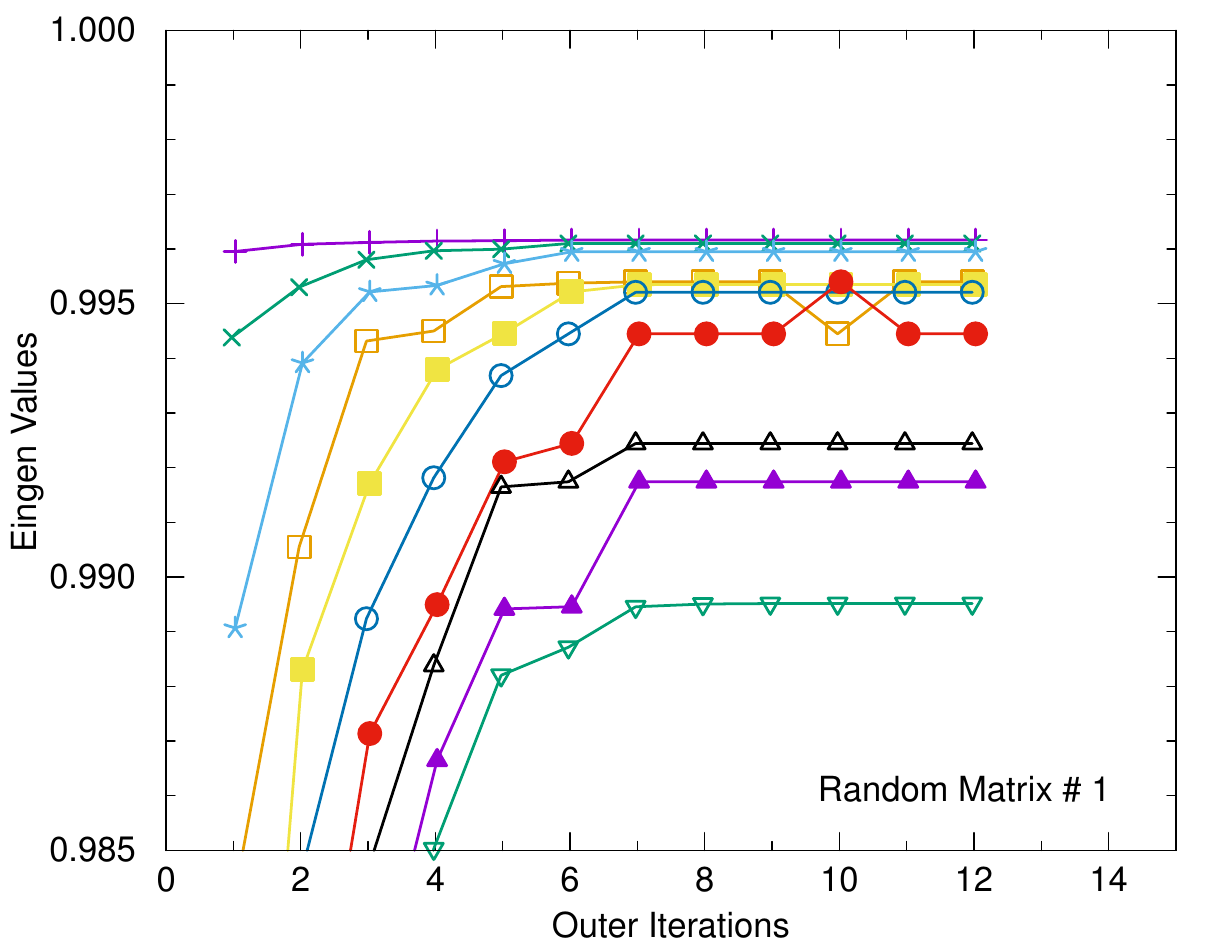}
    \subcaption{TRLAN--JSYM}
    \label{fig:CAs1AevconvTRLANJSYM}
    \end{minipage}
    \begin{minipage}[b]{\figwidth\linewidth}
    \centering
    \includegraphics[clip,trim=0 0 0 0,scale=\figscale]{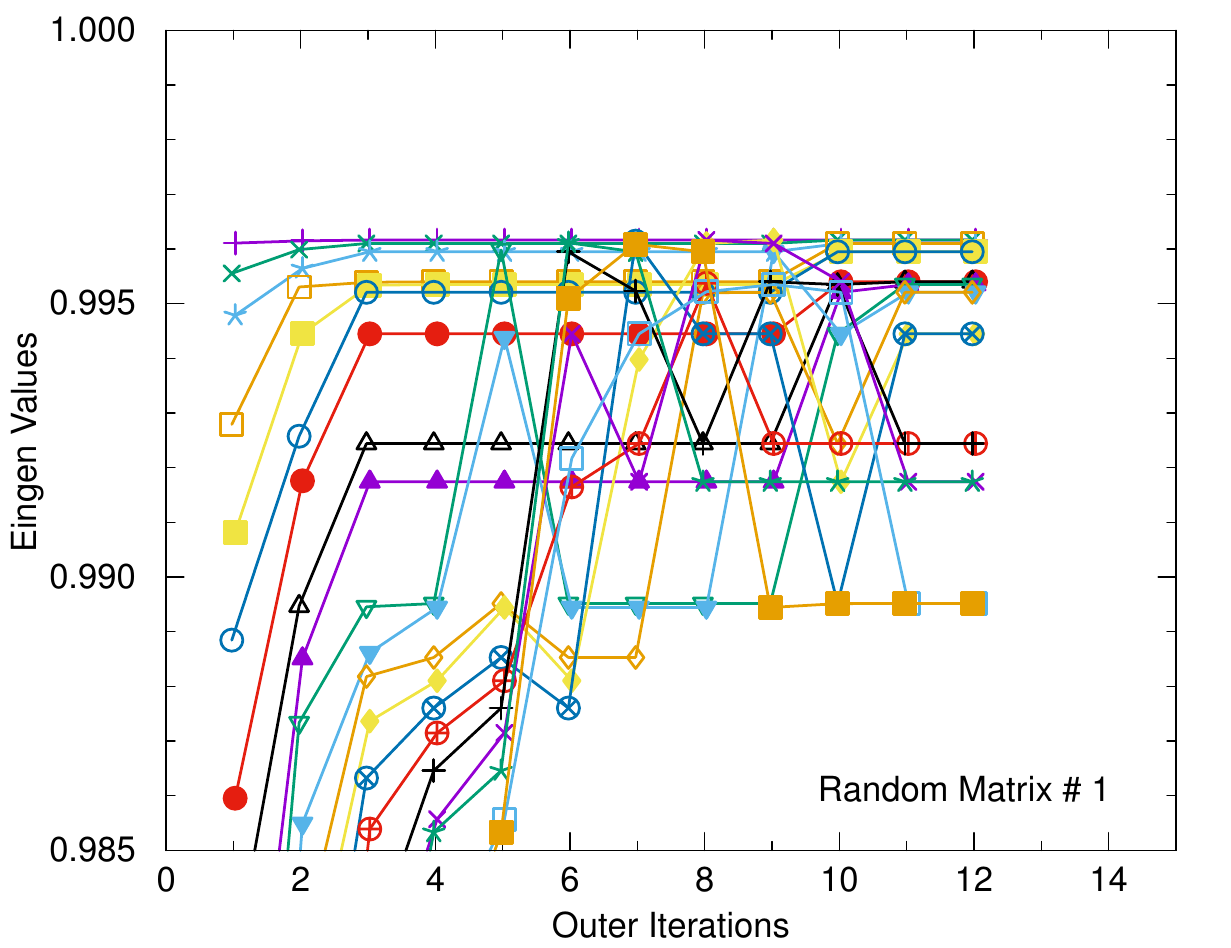}
    \subcaption{TRLAN}
    \label{fig:CAs1AevconvTRLAN}
    \end{minipage}
    \caption{Convergence behavior of the large eigenvalues from the random matrix \# 1 (Case A)}
    \label{fig:CAs1AevconvALL}
\end{figure}

\begin{figure}[t]
    \centering
    \begin{minipage}[b]{\figwidth\linewidth}
    \centering
    \includegraphics[clip,trim=0 0 0 0,scale=\figscale]{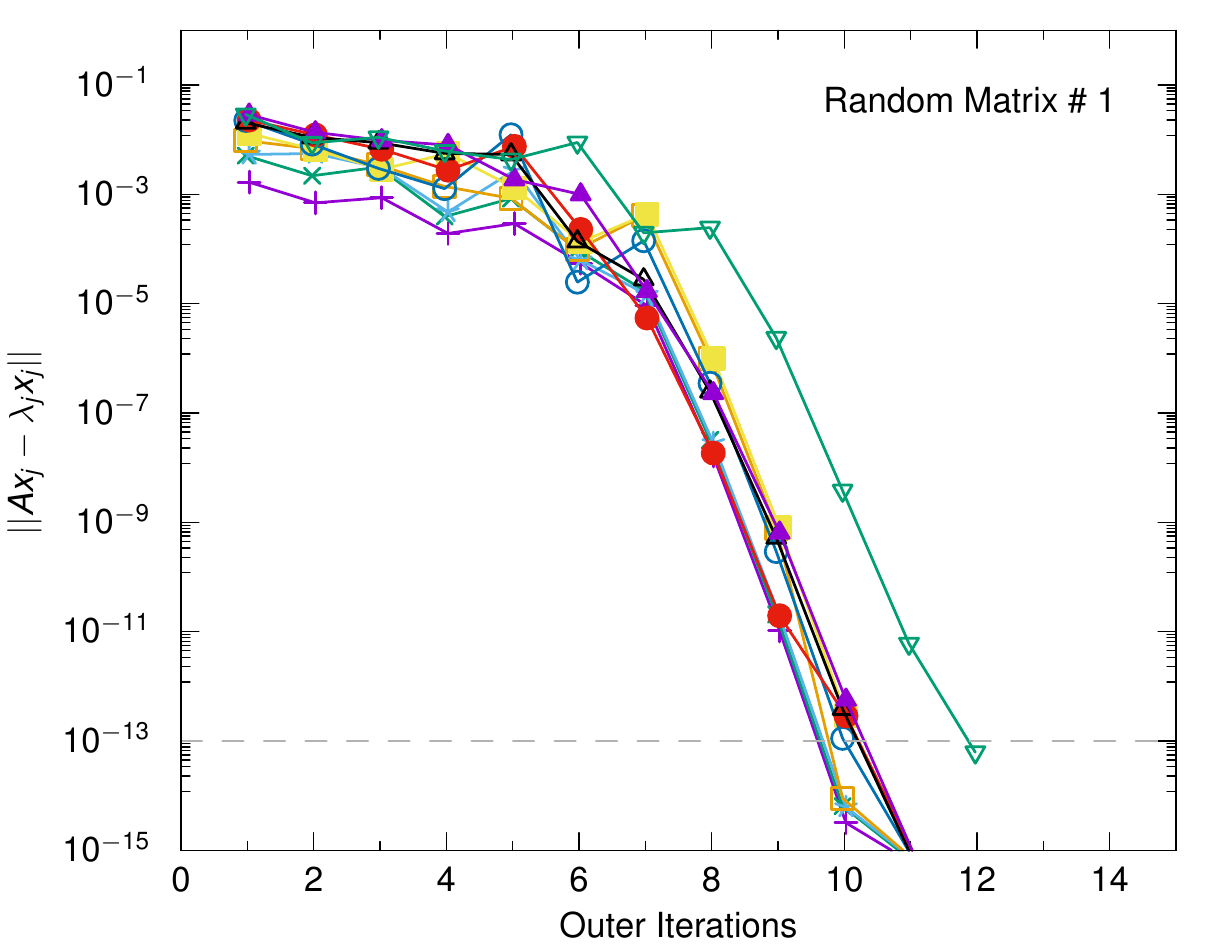}
    \subcaption{TRLAN--JSYM}
    \label{fig:CAs1AerrTRLANJSYM}
    \end{minipage}
    \begin{minipage}[b]{\figwidth\linewidth}
    \centering
    \includegraphics[clip,trim=0 0 0 0,scale=\figscale]{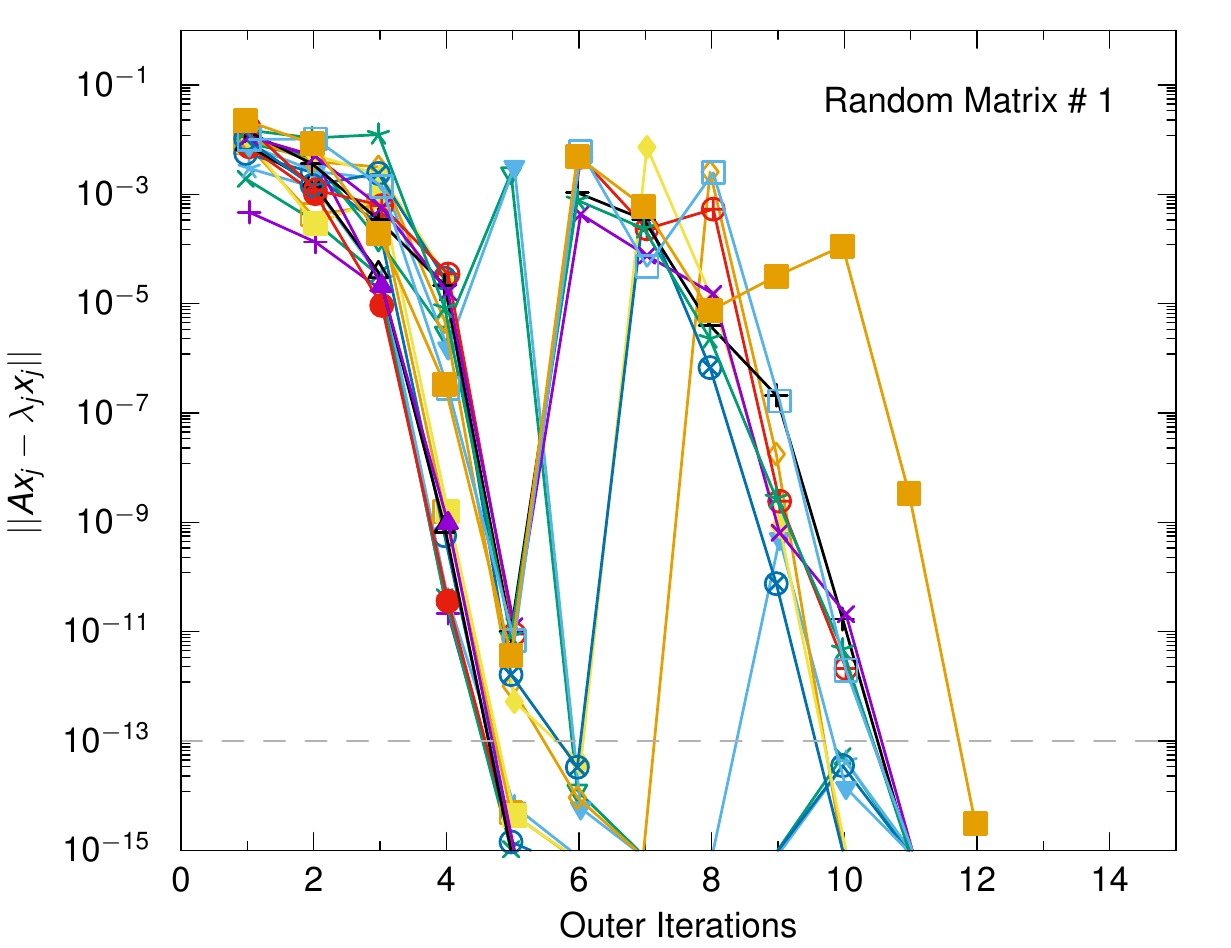}
    \subcaption{TRLAN}
    \label{fig:CAs1AerrTRLAN}
    \end{minipage}
    \caption{Residual history for the large eigenvalues from the random matrix \# 1 (Case A)}
    \label{fig:CAs1AerrALL}
\end{figure}

We compare the computational cost of the algorithms according to the discussion done in the previous section.
The natural parameter choice for the TRLAN algorithm is to employ numbers twice as large as those of the TRLAN--JSYM algorithm. 
The $N_{\mathrm{conv}}$ are comparable among algorithms paired with the doubled parameters, revealing that
$N_{\mathrm{conv}}\simeq N'_{\mathrm{conv}}$ holds for \eqref{eq:MVcost}
and \eqref{eq:MVcostN}. The number of matrix-vector multiplications of
the TRLAN--JSYM algorithm is smaller by roughly a factor of two than that
of the TRLAN algorithm, as seen in the table. Even with the same maximum
Krylov dimension $m$ for both algorithms, \textit{e.g.}
the TRLAN--JSYM with$(\mathrm{nev},\mathrm{mwin},m)=(5,10,100)$ and the TRLAN with $(10,20,100)$, 
the TRLAN--JSYM algorithm still beats the TRLAN algorithm because $m$
does not drastically change the number of matrix-vector multiplications.

\subsection{Case B}

\subsubsection{Definition of the Test Matrix (Case B)}
We evaluate the proposed algorithm~\ref{alg:TRLANJSYM}
for a matrix that appears in a quantum field theory 
called the twisted Eguchi-Kawai (TEK) model 
with adjoint fermions~\cite{GonzalezArroyo:1982ub,GonzalezArroyo:1982hz,Gonzalez-Arroyo:2013bta}.
The equation of motion for the adjoint fermions follows from a matrix $D$
called the Wilson--Dirac operator in the physics literature. The matrix $D=(D_{i,j})$ is defined by
\begin{align}
D_{i,j} &=  \delta_{\alpha,\beta} \delta_{a,b} - 
\kappa \sum_{\mu=1}^{4}
\left[ 
 \left(\delta_{\alpha,\beta} - (\gamma_{\mu})_{\alpha,\beta}\right) (V_{\mu})_{a,b}
+\left(\delta_{\alpha,\beta} + (\gamma_{\mu})_{\alpha,\beta}\right) (V_{\mu})_{b,a} \right],
\label{eq:TEKWDmatrix}
\end{align}
where 
$i$ and $j$ are collective indices of $i=(a,\alpha)$ and $j =(b,\beta)$, respectively.
$V_\mu$ are $(N^2-1)\times (N^2-1)$ matrices 
satisfying $V_{\mu}^{\TP}=V_{\mu}^{\HC}$, \textit{i.e.} real orthonormal matrices
in the adjoint representation of the SU($N$) group,
and $a,b$ denote the group indices running in $1,\dots,N^2-1$. 
$\gamma_\mu$ denotes $4\times 4$ Hermitian matrices satisfying the anti-commuting relation 
$\{\gamma_\mu,\gamma_\nu\} = 2 \delta_{\mu,\nu}I$, and
$\alpha,\beta$ denote spin indices running from $1$ to $4$.
An explicit form of $\gamma_\mu$ can be seen in~\cite{doi:10.1142/8229}.
The parameter $\kappa$ implicitly determines the mass of the fermion.
For more details of $D$, we refer to~\cite{Gonzalez-Arroyo:2013bta,Montvay:2001aj}.

The matrix $D$ satisfies the following properties:
\begin{align}
  \gamma_5 D \gamma_5 &= D^{\HC},
\label{eq:G5Hermit}
\\
         C D C^{\TP}  &= D^{\TP},
\label{eq:MayoranaProp}
\end{align}
where $\gamma_5 = \gamma_4\gamma_1\gamma_2\gamma_3$ and $C=\gamma_4 \gamma_2$.
The matrices $\gamma_5$ and $C$ act only on the spin indices in this notation. 
We employ the definition for $\gamma_\mu$ from~\cite{doi:10.1142/8229}
and give the explicit form in Appendix~\ref{sec:appendix}.
In this case, $\gamma_5$ is real and symmetric and $C$ is real and skew symmetric $C^{\TP}=-C$.
Monte Carlo methods have been used for simulating quantum field theories.
For the system considered herein, the quantum field $V_\mu$ corresponds 
to the stochastic variable in a Monte Carlo algorithm.
The spectrum of $D$ becomes stochastic because it depends on $V_\mu$.

We test the proposed algorithm for the matrix $A$ defined by
\begin{align}
  A \equiv (D \gamma_5)^2 = D D^{\HC}.
\label{eq:TEKmatrix}
\end{align}
The matrix $A$ is Hermitian and $J$-symmetric with $J = C\gamma_5$.
The distribution of the small eigenvalues of $A$ is physically important because
it carries the information about the dynamics of the theory.
The details of the algebraic property of $D$ and $A$ are given in Appendix \ref{sec:appendix}.

\subsubsection{Numerical Results (Case B)}

We set $N=289$ of SU($N$) for the test. The dimension of $A$ is $4\times(289^2-1)=334080$.  
The ensemble for $V_\mu$ is generated with a Monte Carlo algorithm at a parameter set of the TEK model.
We employ a single Monte Carlo sample of $V_\mu$ for the test.

\setlength{\tabcolsep}{4.5pt}
\begin{table}[t]
\centering
    \caption{Algorithm parameters and statistics for convergence (Case B)}
    \label{tab:algparmsB}
    \begin{tabular}{lS[table-format=2.0,table-column-width=2em]
                     S[table-format=2.0,table-column-width=2em]
                     S[table-format=2.0,table-column-width=2em]
                     S[table-format=3.0,table-column-width=3em]
                     S[table-format=3.0,table-column-width=3em]
                     S[table-format=4.1,table-column-width=3em]
                     S[table-format=3.0,table-column-width=3em]
                     S[table-format=3.0,table-column-width=3em]
                     S[table-format=4.1,table-column-width=3em]}
\hline\noalign{\smallskip}
      & & & &
\multicolumn{3}{c}{Large eigenvalues} &
\multicolumn{3}{c}{Small eigenvalues} \\
\cline{5-10}\noalign{\smallskip}
      & {$\mathrm{nev}$} &
        {$\mathrm{mwin}$} & 
        {$m$} & 
        {$N_{\mathrm{conv}}$} &
        {$N_{\mathrm{MV}}$} &
        {\begin{tabular}{wl{3em}}
           Time\\{[sec]}
        \end{tabular}} &
        {$N_{\mathrm{conv}}$} &
        {$N_{\mathrm{MV}}$} &
        {\begin{tabular}{wl{3em}}
           Time\\{[sec]}
        \end{tabular}}
\\
\noalign{\smallskip}\hline\noalign{\smallskip}
TRLAN--JSYM &  4 &  8 &  24 &  54 &  817 & 103.1 &  11 & 178 &  403.2 \\
            &    &    &  48 &  17 &  673 & 133.2 &   5 & 204 &  485.4 \\
            &  8 & 16 &  24 & 711 & 1576 & 334.8 &  46 & 219 &  477.1 \\
            &    &    &  48 &  25 &  770 & 173.6 &   7 & 227 &  540.3 \\
            & 16 & 32 &  48 & 145 & 1128 & 352.8 &  41 & 271 &  648.7 \\
            &    &    &  96 &  16 &  986 & 403.5 &   4 & 276 &  706.9 \\
\noalign{\smallskip}\hline\noalign{\smallskip}
TRLAN       &  8 & 16 &  48 &  42 & 1215 & 154.2 &  11 & 335 &  763.6 \\
            &    &    &  96 &  15 & 1170 & 228.7 &   4 & 327 &  779.3 \\
            & 16 & 32 &  48 & 628 & 2087 & 512.5 & 104 & 441 &  940.5 \\
            &    &    &  96 &  24 & 1408 & 316.0 &   6 & 383 &  915.3 \\
            & 32 & 64 &  96 & 337 & 2332 & 866.5 &  90 & 525 & 1281.4 \\
            &    &    & 192 &  16 & 1881 & 782.3 &   4 & 528 & 1359.7 \\
\noalign{\smallskip}\hline
    \end{tabular}
\end{table}

We compare the convergence behavior of the eigenvalues between the proposed algorithm (TRLAN--JSYM) 
and the standard (single vector) thick-restart Lanczos algorithm (TRLAN).
We use the normal and invert modes for solving large and small eigenvalues, respectively.
The conjugate--gradient (CG) algorithm is used in the invert mode.
The algorithmic parameters, the number of desired eigenvalues $\mathrm{nev}$, 
the restart window size $\mathrm{mwin}$,
and the maximum size of the search dimension $m$ are shown in Table~\ref{tab:algparmsB}.
We also tabulate the results of the outer iteration count, the number of matrix-vector multiplications, and
the computational time for the convergence.
The timings are shown as reference values, showing how the cost of the matrix-vector multiplication dominates 
the computational time in actual applications.
We note that the convergence behavior of the CG in the invert mode
was almost identical between the two algorithms as has been assumed in Section~\ref{sec:3},
justifying the cost comparison in terms of the number of matrix-vector multiplications of $A^{-1}v_j$ in the invert mode.
Compared with the TRLAN--JSYM,
we double the parameters of the TRLAN to 
find all doubly degenerate eigenvalues.
We set the tolerance to be $10^{-13}$ for the eigensolvers.

Figs.~\ref{fig:AevconvALL} and~\ref{fig:AerrALL} 
show the convergence behavior and residual history of the large eigenvalues, respectively.
The algorithmic parameters are $(\mathrm{nev},\mathrm{mwin},m)=(8,16,48)$
for the TRLAN--JSYM and $(16,32,96)$ for the TRLAN, respectively.
We observe similar convergence behavior as in Case A,
where several reorderings occur among approximate eigenvalues during the iterations in the TRLAN algorithm.
The same convergence behavior is seen 
in Figs.~\ref{fig:invAevconvALL} and \ref{fig:invAerrALL} for the small eigenvalues.

\begin{figure}[t]
    \centering
    \begin{minipage}[b]{\figwidth\linewidth}
    \centering
    \includegraphics[clip,trim=0 0 0 0,scale=\figscale]{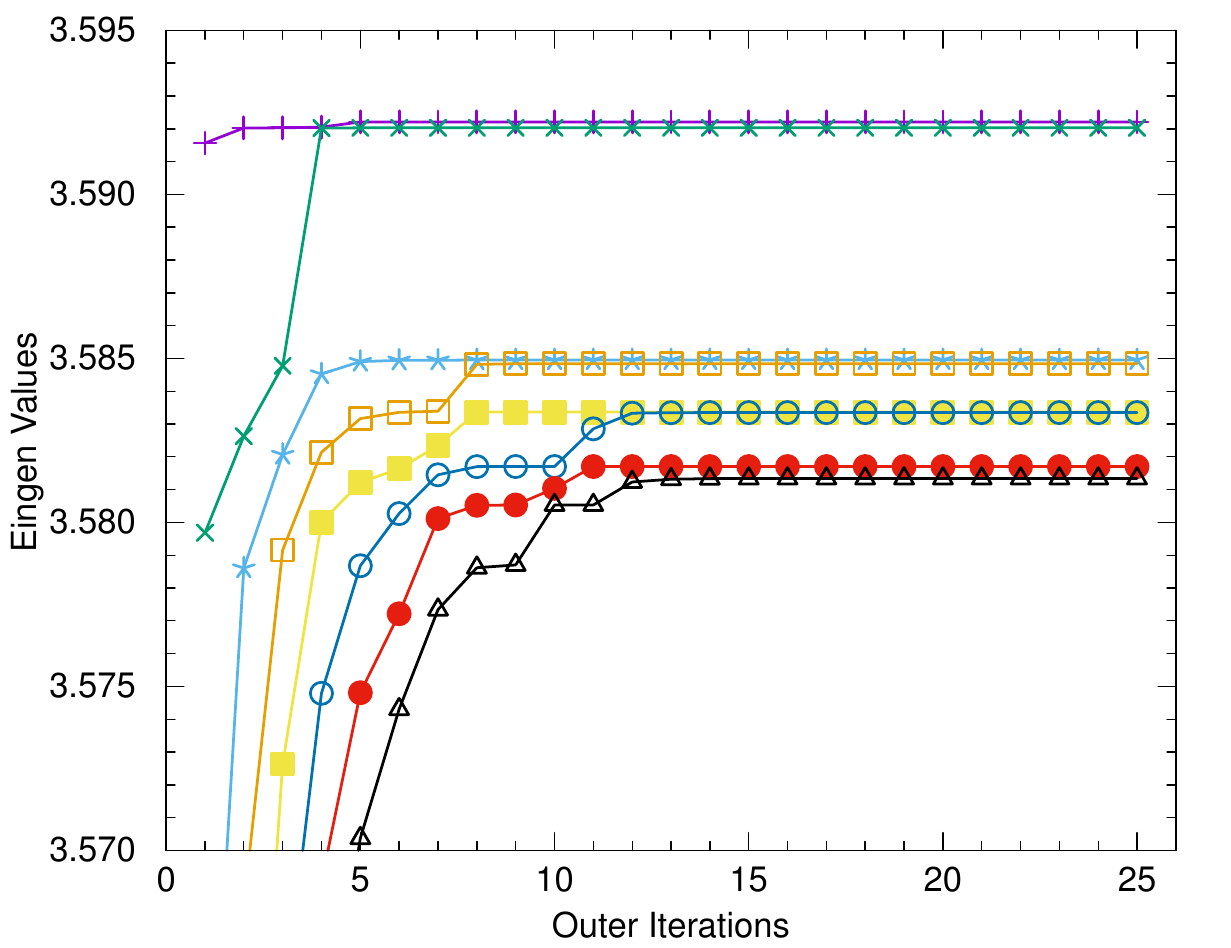}        
    \subcaption{TRLAN--JSYM}
    \label{fig:AevconvTRLANJSYM}
    \end{minipage}
    \begin{minipage}[b]{\figwidth\linewidth}
    \centering
    \includegraphics[clip,trim=0 0 0 0,scale=\figscale]{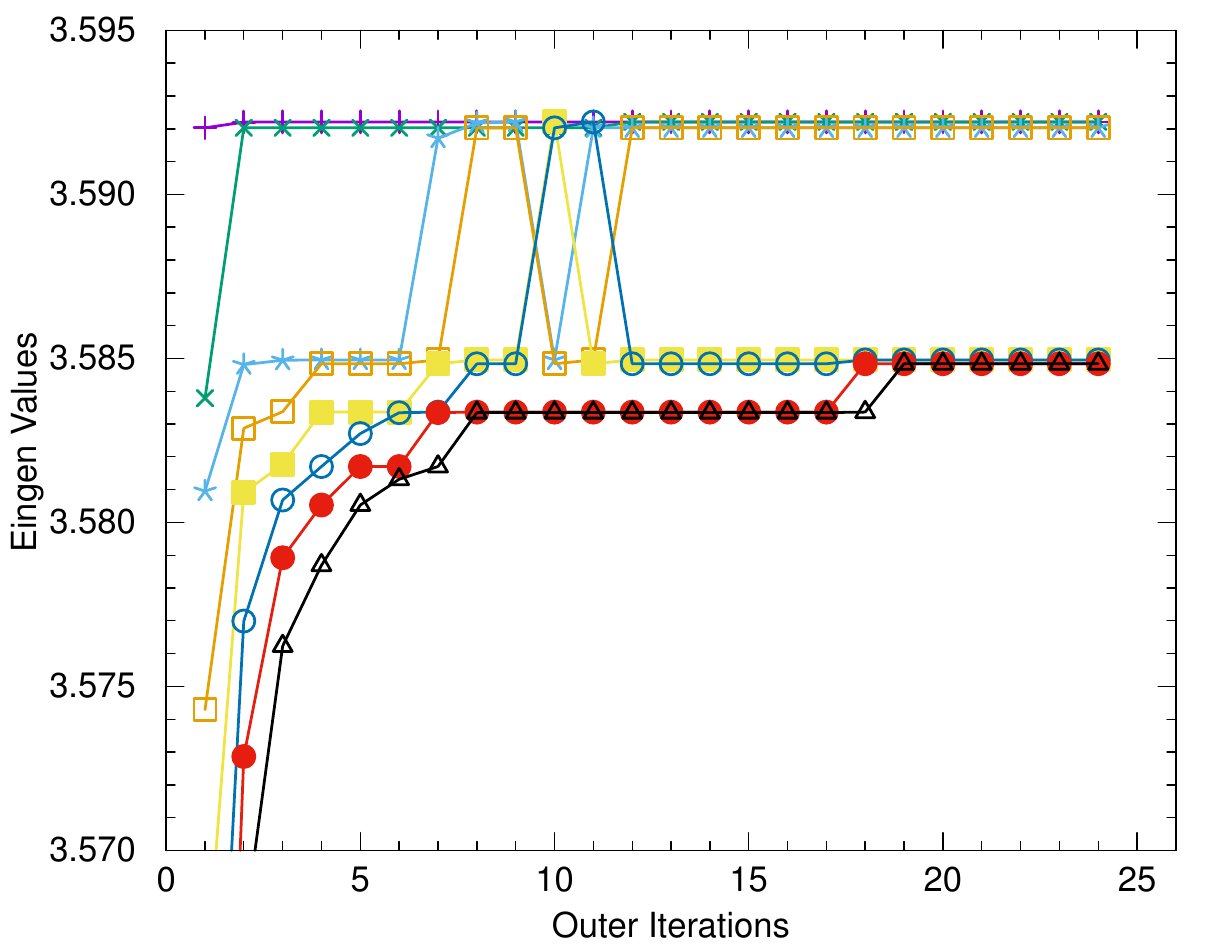}
    \subcaption{TRLAN}
    \label{fig:AevconvTRLAN}
    \end{minipage}
    \caption{Convergence behavior of the large eigenvalues (Case B)}
    \label{fig:AevconvALL}
\end{figure}
\begin{figure}[t]
    \centering
    \begin{minipage}[b]{\figwidth\linewidth}
    \centering
    \includegraphics[clip,trim=0 0 0 0,scale=\figscale]{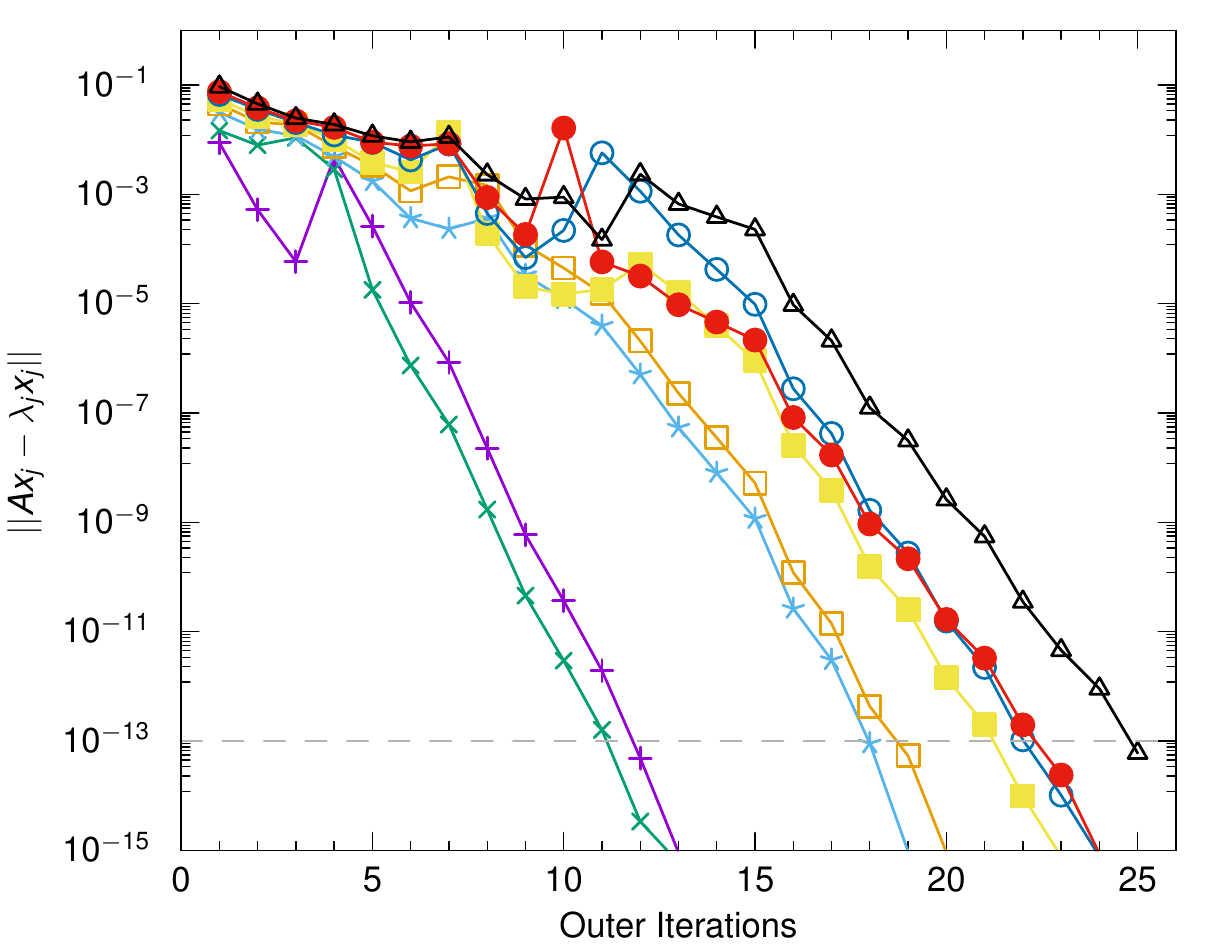}        
    \subcaption{TRLAN--JSYM}
    \label{fig:AerrTRLANJSYM}
    \end{minipage}
    \begin{minipage}[b]{\figwidth\linewidth}
    \centering
    \includegraphics[clip,trim=0 0 0 0,scale=\figscale]{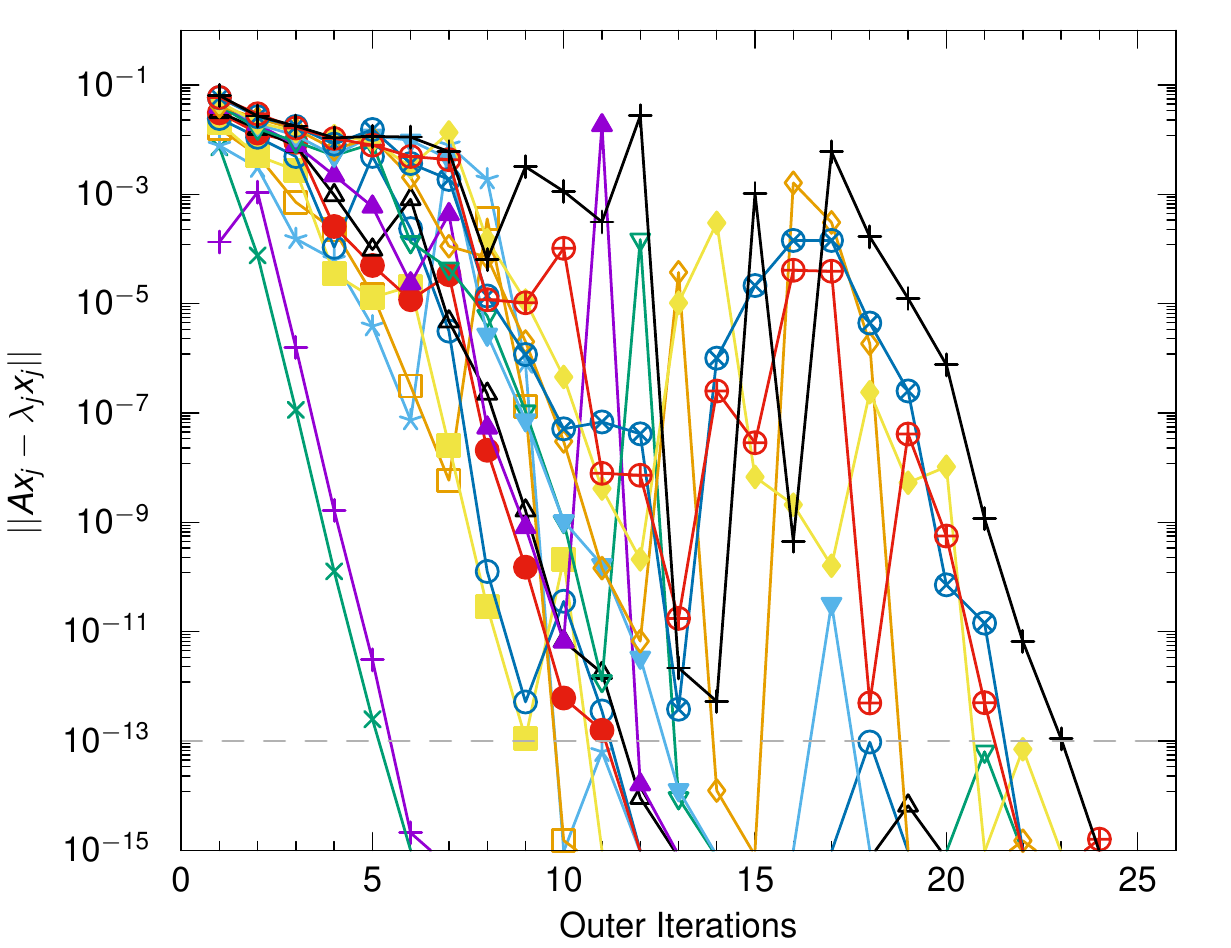}
    \subcaption{TRLAN}
    \label{fig:AerrTRLAN}
    \end{minipage}
    \caption{Residual history for the large eigenvalues (Case B)}
    \label{fig:AerrALL}
\end{figure}

The computational costs are compared in Table~\ref{tab:algparmsB}.
According to the discussion done in Section~\ref{sec:3},
most cases satisfy  $N_{\mathrm{conv}}\simeq N'_{\mathrm{conv}}$ for \eqref{eq:MVcost} and \eqref{eq:MVcostN}
among algorithms paired with doubled parameters,
and $N_{\mathrm{MV}}$ for the TRLAN--JSYM algorithm is approximately twice as small as that for the TRLAN algorithm.
Three cases, $(\mathrm{nev},\mathrm{mwin},m)=(16,32,48)$ for both modes, and $(\mathrm{nev},\mathrm{mwin},m)=(8,16,24)$ 
for the invert mode, are the exceptions. 
In these cases, one or two eigenvalues, 
which are the largest for the invert mode or the smallest for the normal mode among $\mathrm{nev}$ eigenvalues,
show slow convergence.
Even for these cases, however, the TRLAN--JSYM shows smoother convergence behavior than that of the TRLAN.
All the cases we have investigated show that the TRLAN--JSYM algorithm 
has better computational cost than that of TRLAN regarding $N_{\mathrm{MV}}$.
The computational timings are roughly proportional to $N_{\mathrm{MV}}$,
indicating that the matrix-vector multiplication dominates the timings.
With the invert mode for small eigenvalue problems, 
the timings are well proportional to the number of matrix-vector multiplications,
because the CG algorithm is used for $A^{-1}v$ and the cost of the Lanczos and
the true residual computing parts become negligible.

\begin{figure}[t]
    \centering
    \begin{minipage}[b]{\figwidth\linewidth}
    \centering
    \includegraphics[clip,trim=0 0 0 0,scale=\figscale]{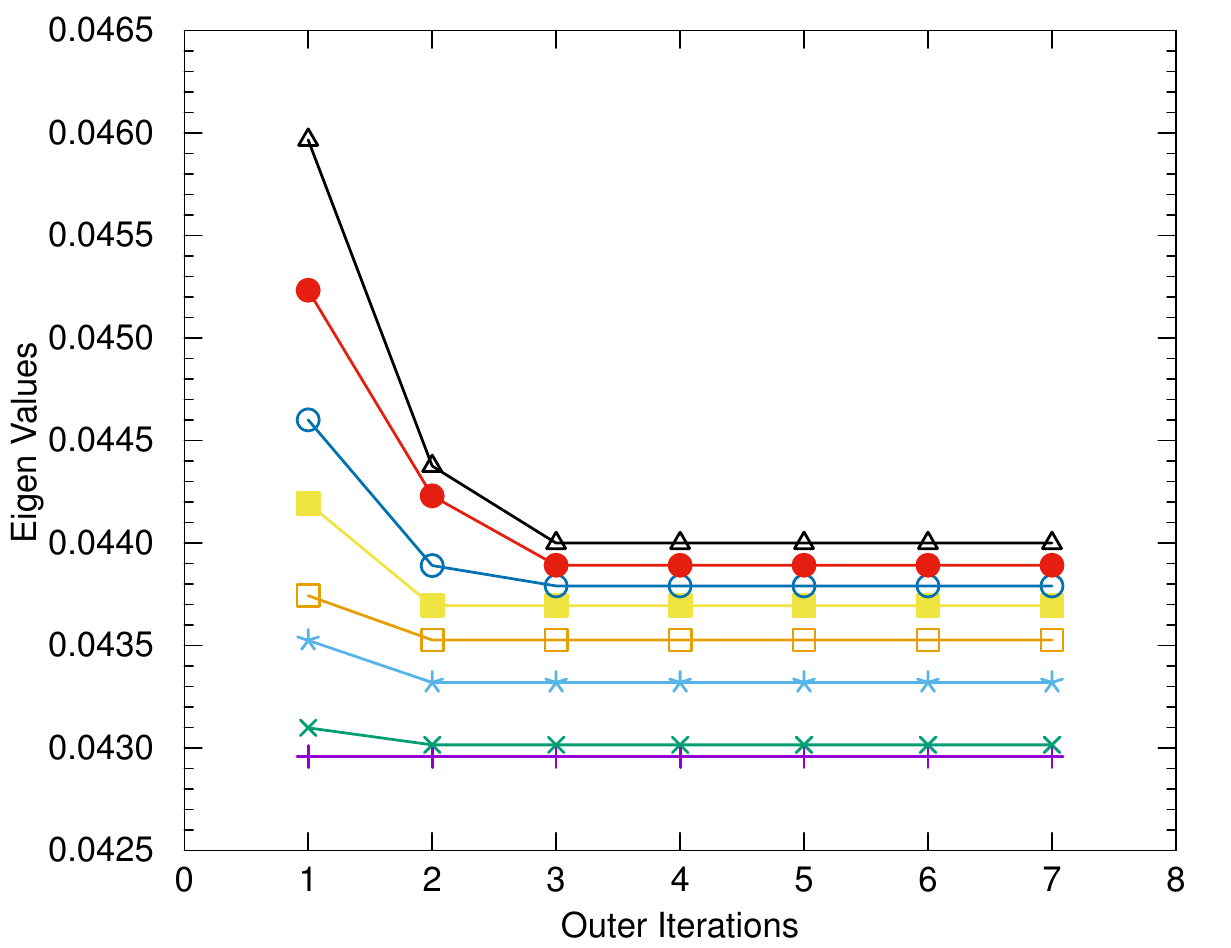}
    \subcaption{TRLAN--JSYM}
    \label{fig:invAevconvTRLANJSYM}
    \end{minipage}
    \begin{minipage}[b]{\figwidth\linewidth}
    \centering
    \includegraphics[clip,trim=0 0 0 0,scale=\figscale]{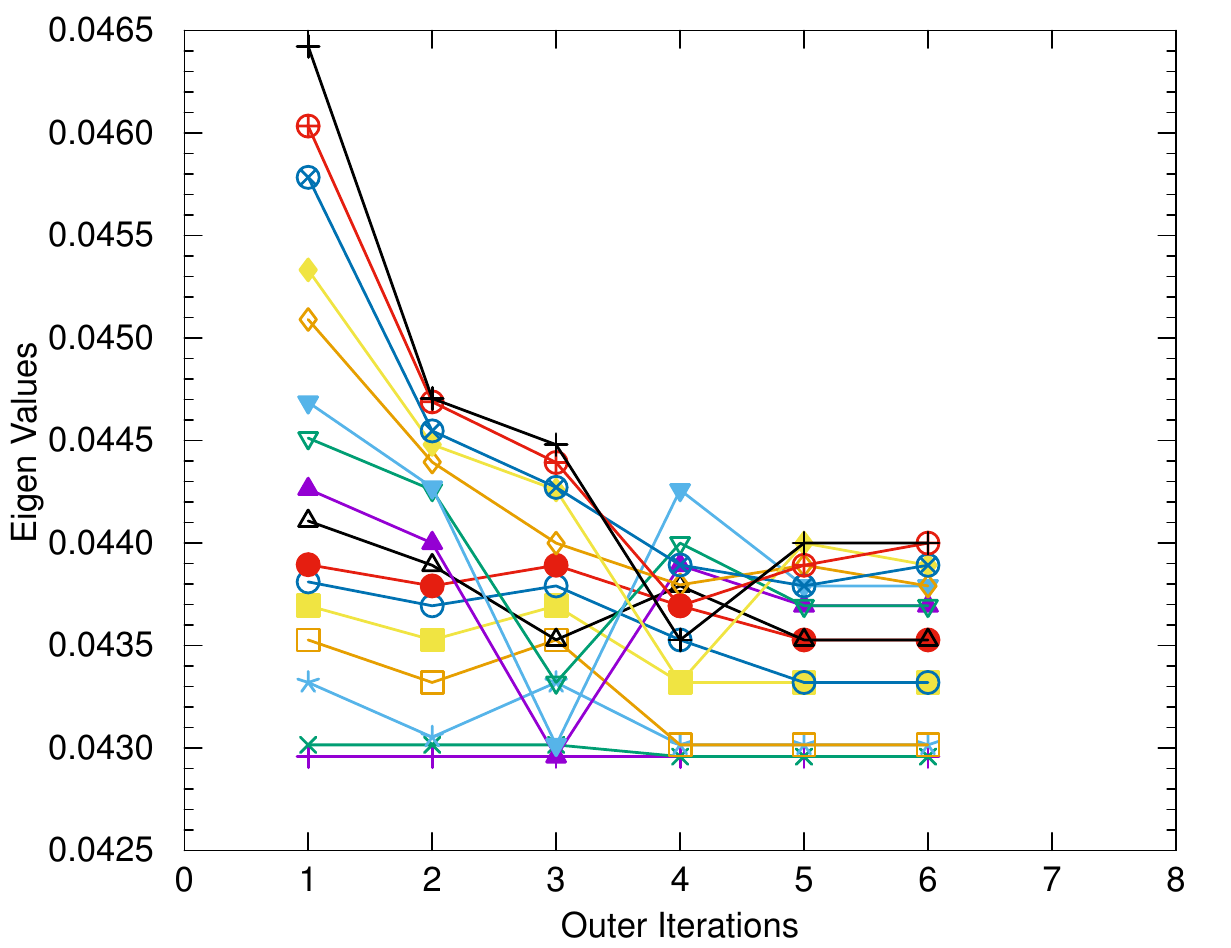}
    \subcaption{TRLAN}
    \label{fig:invAevconvTRLAN}
    \end{minipage}
    \caption{Convergence behavior of the small eigenvalues (Case B)}
    \label{fig:invAevconvALL}
\end{figure}
\begin{figure}[t]
    \centering
    \begin{minipage}[b]{\figwidth\linewidth}
    \centering
    \includegraphics[clip,trim=0 0 0 0,scale=\figscale]{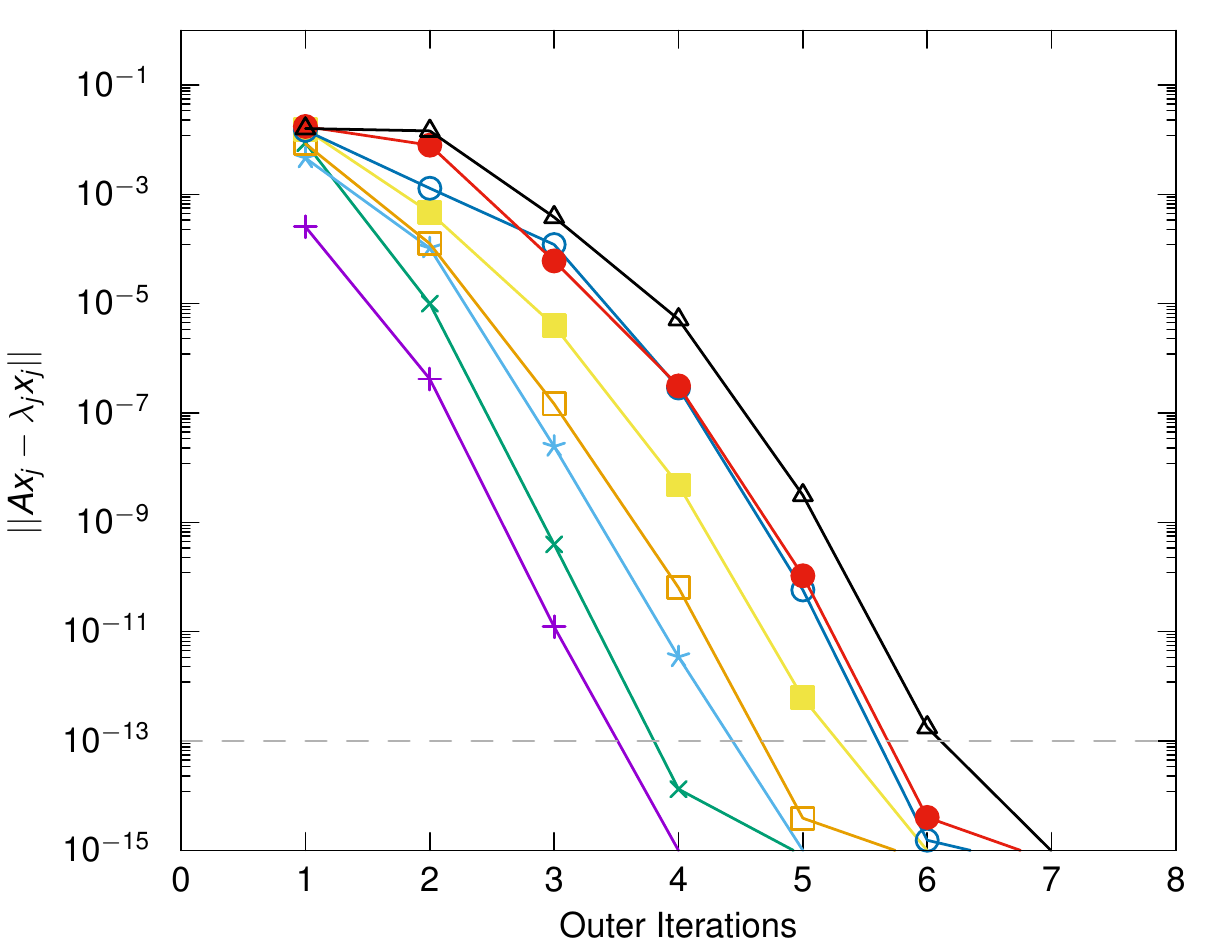}
    \subcaption{TRLAN--JSYM}
    \label{fig:invAerrTRLANJSYM}
    \end{minipage}
    \begin{minipage}[b]{\figwidth\linewidth}
    \centering
    \includegraphics[clip,trim=0 0 0 0,scale=\figscale]{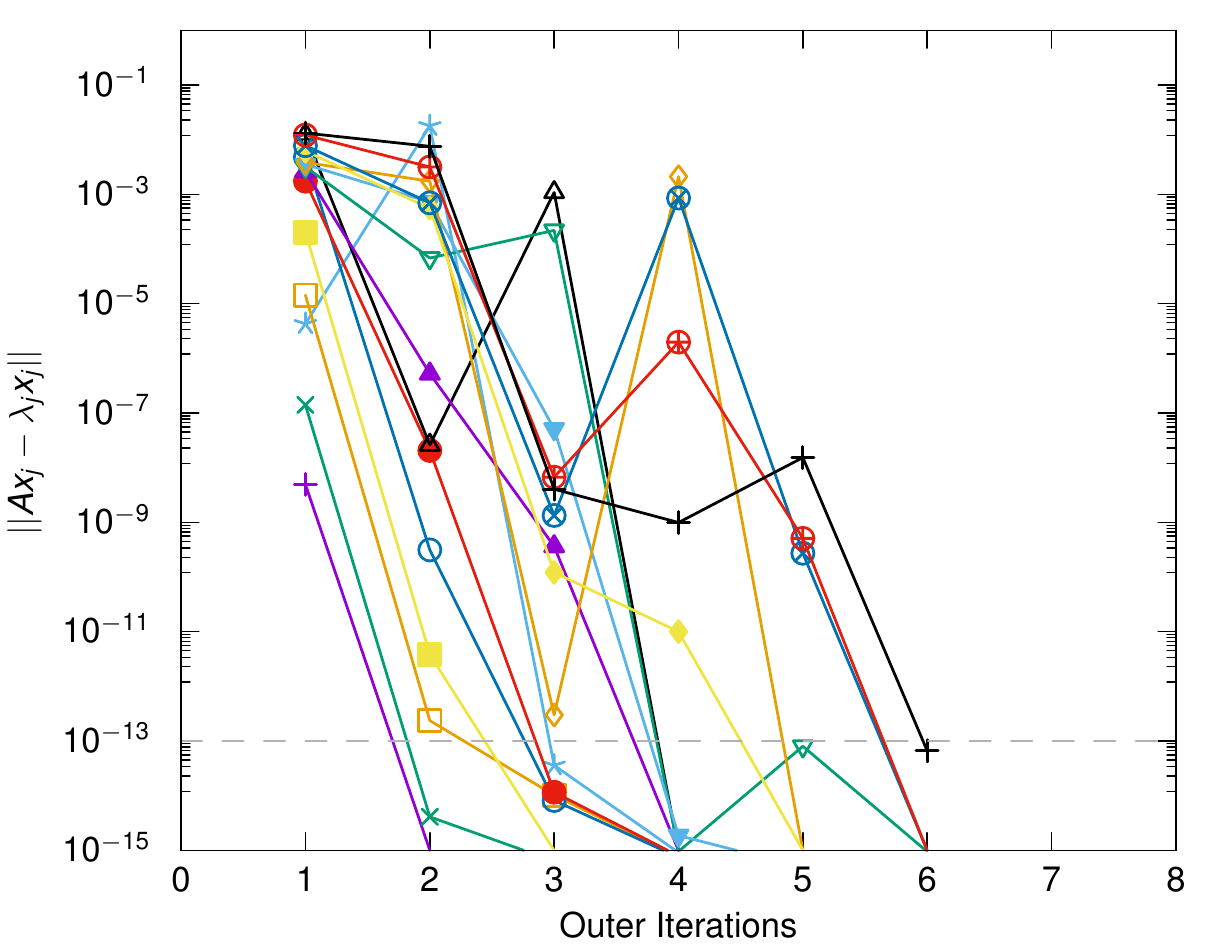}
    \subcaption{TRLAN}
    \label{fig:invAerrTRLAN}
    \end{minipage}
    \caption{Residual history for the small eigenvalues (Case B)}
    \label{fig:invAerrALL}
\end{figure}

\section{Summary}
\label{sec:5}

In this study, we have shown the orthogonality and $A$-orthogonality between the two Krylov subspaces,
$\mathcal{K}_k(A,v)$ and $\mathcal{K}_k(A,Jv^*)$ 
that are generated with the Lanczos algorithm for Hermitian $J$-symmetric matrices $A$.
By employing this property, we proposed the thick-restarted Lanczos algorithm for
Hermitian $J$-symmetric matrices (TRLAN--JSYM)
using which we could efficiently search for one half of the doubly degenerate eigenvectors in $\mathcal{K}_k(A,v)$ 
without the need to explicitly construct $\mathcal{K}_k(A,Jv^*)$.
The other half of the degenerate eigenvectors are simply constructed 
from the converged eigenvectors by utilizing the $J$-symmetry property.

We demonstrated the proposed algorithm TRLAN--JSYM 
for two test cases, the random matrices, and the fermion matrix from the quantum field theory called the TEK model.
The convergence observed for the TRLAN--JSYM algorithm was smoother than that for the TRLAN algorithm, as expected.
The TRLAN--JSYM algorithm performed better than the TRLAN algorithm regarding the matrix-vector multiplication.
The TRLAN algorithm shows the reordering of eigenvalues among the degenerated eigenvalues
caused by the loss of orthogonality between 
the eigenvectors paired with $J$-symmetry during the standard Lanczos iteration
with the finite precision arithmetic.
We did not discuss the mathematical background on the loss of orthogonality in the Lanczos algorithm.
If exact arithmetic was employed for both of the algorithms, 
the TRLAN algorithm becomes identical to the TRLAN--JSYM algorithm, according to Corollary~\ref{coro1}.
However, our algorithm enforces the orthogonality to achieve the smooth convergence behavior
at finite precision arithmetic, resulting in better performance.

%


\begin{acknowledgements}
Numerical computations are performed on the ITO supercomputer system of Kyushu university.
This work is partly supported by Priority Issue 9 to be tackled by using Post K Computer. 
We thank Antonio Gonz\'{a}lez-Arroyo and Masanori Okawa for comments on the draft and 
the details of the Wilson-Dirac operator in the adjoint representation.
We are very grateful to the anonymous reviewer for his/her comments that enhanced the quality of the manuscript.
\end{acknowledgements}

\appendix
\section{Properties of the matrices \texorpdfstring{$D$}{D} \eqref{eq:TEKWDmatrix} and \texorpdfstring{$A$}{A} \eqref{eq:TEKmatrix}}
\label{sec:appendix}
In this appendix we show the algebraic properties of 
the matrices $D$ \eqref{eq:TEKWDmatrix} and $A$ \eqref{eq:TEKmatrix}, including the $J$-symmetry. 
We employ the following explicit form for $\gamma_\mu$:
\begin{align}
  \gamma_1 &=
  \begin{bmatrix*}[r]
      0 &    0 & 0 &-i \\
      0 &    0 &-i & 0 \\
      0 &\pp i & 0 & 0 \\
      i &    0 & 0 & 0
  \end{bmatrix*},\quad
  \gamma_2 =
  \begin{bmatrix*}[r]
      0 &    0 &    0 &-1 \\
      0 &    0 &\pp 1 & 0 \\
      0 &\pp 1 &    0 & 0 \\
     -1 &    0 &    0 & 0
  \end{bmatrix*},\quad
  \gamma_3 =
  \begin{bmatrix*}[r]
      0 & 0 &-i & 0 \\
      0 & 0 & 0 &\pp i \\
      i & 0 & 0 & 0 \\
      0 &-i & 0 & 0
  \end{bmatrix*},\quad
  \gamma_4 =
  \begin{bmatrix*}[r]
      1 &    0 & 0 & 0 \\
      0 &\pp 1 & 0 & 0 \\
      0 &    0 &-1 & 0 \\
      0 &    0 & 0 &-1
  \end{bmatrix*}.
\end{align}
In addition to these, we also have $\gamma_5=\gamma_4\gamma_1\gamma_2\gamma_3$ as
\begin{align}
  \gamma_5 =
  \begin{bmatrix*}[r]
      0 &    0 &\pp 1 &    0 \\
      0 &    0 & 0    &\pp 1 \\
      1 &    0 & 0    &    0 \\
      0 &\pp 1 & 0    &    0
  \end{bmatrix*}.
\end{align}
These $\gamma_\mu$ matrices have the following properties:
\begin{alignat}{2}
  \{\gamma_\mu,\gamma_\nu\} &= 2\delta_{\mu,\nu}I, &\quad&\mbox{for $\mu,\nu=1,\dots,5$},\\
           \gamma_\mu^{\HC} &=  \gamma_\mu,        &\quad&\mbox{for $\mu=1,\dots,5$},\\
             \gamma_1^{\TP}  = -\gamma_1,\quad
             \gamma_2^{\TP} &=  \gamma_2,\quad
             \gamma_3^{\TP} &=&-\gamma_3,\quad
             \gamma_4^{\TP}  =  \gamma_4,\quad
             \gamma_5^{\TP}  =  \gamma_5.
\end{alignat}
The matrix $C=\gamma_4\gamma_2$ has the following properties:
\begin{align}
  C^{-1} &= \gamma_2\gamma_4 = - \gamma_4\gamma_2 = -C,\\
  C^{-1} &= \gamma_2\gamma_4 =   \gamma_2^{\TP}\gamma_4^{\TP} = \left(\gamma_4\gamma_2\right)^{\TP} = C^{\TP},\\
 C\gamma_\mu C^{\TP} &= 
\left.
   \begin{cases}
       \pp \gamma_4\gamma_2 \gamma_1\gamma_2\gamma_4 =  -\gamma_4\gamma_1\gamma_4 = \pp \gamma_1 = -\gamma_1^{\TP} &(\mu=1)\\
       \pp \gamma_4\gamma_2 \gamma_2\gamma_2\gamma_4 =\pp\gamma_4\gamma_2\gamma_4 =   - \gamma_2 = -\gamma_2^{\TP} &(\mu=2)\\
       \pp \gamma_4\gamma_2 \gamma_3\gamma_2\gamma_4 =  -\gamma_4\gamma_3\gamma_4 = \pp \gamma_3 = -\gamma_3^{\TP} &(\mu=3)\\
       \pp \gamma_4\gamma_2 \gamma_4\gamma_2\gamma_4 =  -\gamma_4\gamma_4\gamma_4 =   - \gamma_4 = -\gamma_4^{\TP} &(\mu=4)
   \end{cases}
\right\} = -\gamma_\mu^{\TP}.
\label{eq:CCGamma}
\end{align}

We show the properties of \eqref{eq:G5Hermit} and \eqref{eq:MayoranaProp} in detail. 
To simplify the proof, we suppress the matrix indices of \eqref{eq:TEKWDmatrix} 
and write it as
\begin{align}
  D=I - \kappa \sum_{\mu=1}^{4}\left[ (1-\gamma_\mu)V_\mu + (1+\gamma_\mu)V_\mu^{\TP}\right],
\end{align}
where the direct product of the spinor index and the color index is implicit.

Equation \eqref{eq:G5Hermit} is shown as:
\begin{align}
  \gamma_5 D\gamma_5 &=
\gamma_5\left(I - \kappa \sum_{\mu=1}^{4}\left[ (1-\gamma_\mu)V_\mu + (1+\gamma_\mu)V_{\mu}^{\TP}\right]\right)\gamma_5
\notag\\
&=
I - \kappa \sum_{\mu=1}^{4}\left[ \gamma_5(1-\gamma_\mu)\gamma_5V_\mu + \gamma_5(1+\gamma_\mu)\gamma_5V_{\mu}^{\TP}\right]
\notag\\
&=
I - \kappa \sum_{\mu=1}^{4}\left[ (1+\gamma_\mu)V_{\mu} + (1-\gamma_\mu)V_{\mu}^{\TP}\right],
\notag
\end{align}
where $\{\gamma_5,\gamma_\mu\}=0$ is used. 
Because $V_\mu^{\TP} = V_\mu^{\HC}$ for the matrices in the adjoint representation of the SU($N$) group, 
the last line is identical to
\begin{align}
&=
\left(I - \kappa \sum_{\mu=1}^{4}\left[ (1-\gamma_\mu)V_{\mu} + (1+\gamma_\mu)V_{\mu}^{\TP}\right]\right)^{\HC}
\notag\\
&= D^{\HC}.
\end{align}

Next, we show \eqref{eq:MayoranaProp} in:
\begin{align}
  C D C^{\TP} 
&=  C\left(I - \kappa \sum_{\mu=1}^{4}\left[ (1-\gamma_\mu)V_\mu + (1+\gamma_\mu)V_{\mu}^{\TP}\right]\right)C^{\TP}
\notag\\
&=  I - \kappa \sum_{\mu=1}^{4}\left[ C(1-\gamma_\mu)C^{\TP}V_\mu + C(1+\gamma_\mu)C^{\TP}V_{\mu}^{\TP}\right]
\notag\\
&=  I - \kappa \sum_{\mu=1}^{4}\left[ (1+\gamma_\mu^{\TP})V_\mu + (1-\gamma_\mu^{\TP})V_{\mu}^{\TP}\right]
\notag,
\end{align}
where we used \eqref{eq:CCGamma}. The last line is identical to
\begin{align}
&=  \left(I - \kappa \sum_{\mu=1}^{4}\left[ (1-\gamma_\mu)V_\mu + (1+\gamma_\mu)V_{\mu}^{\TP}\right]\right)^{\TP}
\notag\\
&= D^{\TP}.
\end{align}

We finally show the $J$-symmetry of $A$ \eqref{eq:TEKmatrix}. 
The Hermiticity of $A$ is apparent from \eqref{eq:G5Hermit} and \eqref{eq:TEKmatrix}.
The properties of $J\equiv C\gamma_5=\gamma_4\gamma_2\gamma_5$ are:
\begin{align}
 J^{-1} &=  \gamma_5\gamma_2\gamma_4 
         = -\gamma_5\gamma_4\gamma_2
         =  \gamma_4\gamma_5\gamma_2
         = -\gamma_4\gamma_2\gamma_5
         = - J,\\
 J^{-1} &= \gamma_5\gamma_2\gamma_4 
         = \gamma_5^{\TP}\gamma_2^{\TP}\gamma_4^{\TP} 
         = \left(\gamma_4\gamma_2\gamma_5\right)^{\TP}
         = J^{\TP}.
\end{align}
Because $\gamma_2, \gamma_4,$ and $\gamma_5$ are real matrices, $J$ is real. 
Therefore, the properties of $J\equiv C\gamma_5$ follows those in \eqref{eq:symmetry}.
The $J$-symmetry of $A$ \eqref{eq:TEKmatrix} is shown as:
\begin{align}
JAJ^{-1} 
 &= C \gamma_5 D \gamma_5 D \gamma_5 \gamma_5 C^{\TP} \notag\\
 &= C \gamma_5 D \gamma_5 D C^{\TP}                   \notag\\
 &= \gamma_5 C D \gamma_5 D C^{\TP}                   \notag\\
 &= \gamma_5 C D C^{\TP} C \gamma_5  D C^{\TP}        \notag\\
 &= \gamma_5 C D C^{\TP} \gamma_5 C  D C^{\TP}        \notag\\
 &= \gamma_5  D^{\TP} \gamma_5 D^{\TP}               \notag\\
 &= (D \gamma_5)^{\TP} (D \gamma_5)^{\TP}            \notag\\
 &= \left[(D \gamma_5)(D \gamma_5)\right]^{\TP}      \notag\\
 &= A^{\TP},
\end{align}
where we used $C\gamma_5=\gamma_5C,  C^{\TP}C = I, \gamma_5^{\TP}=\gamma_5, (\gamma_5)^2=I$, and \eqref{eq:MayoranaProp}.
Therefore, $A$ of \eqref{eq:TEKmatrix} is $J$-symmetric.

%
%





\begin{thebibliography}{10}
\providecommand{\url}[1]{{#1}}
\providecommand{\urlprefix}{URL }
\expandafter\ifx\csname urlstyle\endcsname\relax
  \providecommand{\doi}[1]{DOI~\discretionary{}{}{}#1}\else
  \providecommand{\doi}{DOI~\discretionary{}{}{}\begingroup
  \urlstyle{rm}\Url}\fi

\bibitem{doi:10.1137/S1064827501397949}
Baglama, J., Calvetti, D., Reichel, L.: {IRBL: An Implicitly Restarted
  Block-{Lanczos} Method for Large-Scale {Hermitian} Eigenproblems}.
\newblock SIAM J. Sci. Comput. \textbf{24}(5), 1650--1677 (2003).
\newblock \doi{10.1137/S1064827501397949}.
\newblock \urlprefix\url{https://doi.org/10.1137/S1064827501397949}

\bibitem{doi:10.1137/1.9780898719581}
Bai, Z., Demmel, J., Dongarra, J., Ruhe, A., van~der Vorst, H.: {Templates for
  the Solution of Algebraic Eigenvalue Problems}.
\newblock Society for Industrial and Applied Mathematics (2000).
\newblock \doi{10.1137/1.9780898719581}.
\newblock
  \urlprefix\url{https://epubs.siam.org/doi/abs/10.1137/1.9780898719581}

\bibitem{BENNER199775}
Benner, P., Fa{\ss}bender, H.: {An implicitly restarted symplectic {Lanczos}
  method for the {Hamiltonian} eigenvalue problem}.
\newblock Linear Algebra Appl. \textbf{263}, 75 -- 111 (1997).
\newblock \doi{https://doi.org/10.1016/S0024-3795(96)00524-1}.
\newblock
  \urlprefix\url{http://www.sciencedirect.com/science/article/pii/S0024379596005241}

\bibitem{BENNER2011578}
Benner, P., Fa{\ss}bender, H., Stoll, M.: A Hamiltonian {Krylov}--{Schur}-type
  method based on the symplectic Lanczos process.
\newblock Linear Algebra Appl. \textbf{435}(3), 578 -- 600 (2011).
\newblock \doi{https://doi.org/10.1016/j.laa.2010.04.048}.
\newblock
  \urlprefix\url{http://www.sciencedirect.com/science/article/pii/S0024379510002867}

\bibitem{BENNER2018407}
Benner, P., Fa{\ss}bender, H., Yang, C.: {Some remarks on the complex
  {$J$}-symmetric eigenproblem}.
\newblock Linear Algebra Appl. \textbf{544}, 407 -- 442 (2018).
\newblock \doi{https://doi.org/10.1016/j.laa.2018.01.014}.
\newblock
  \urlprefix\url{http://www.sciencedirect.com/science/article/pii/S0024379518300211}

\bibitem{etna_vol2_pp1-21}
Calvetti, D., Reichel, L., Sorensen, D.C.: {An implicitly restarted {Lanczos}
  method for large symmetric eigenvalue problems}.
\newblock Electron. Trans. Numer. Anal. \textbf{2}, 1--21 (1994)

\bibitem{DONGARRA198427}
Dongarra, J., Gabriel, J., Koelling, D., Wilkinson, J.: The eigenvalue problem
  for Hermitian matrices with time reversal symmetry.
\newblock Linear Algebra Appl. \textbf{60}, 27 -- 42 (1984).
\newblock \doi{https://doi.org/10.1016/0024-3795(84)90068-5}.
\newblock
  \urlprefix\url{http://www.sciencedirect.com/science/article/pii/0024379584900685}

\bibitem{GOLUB1977361}
Golub, G., Underwood, R.: {The Block {Lanczos} Method for Computing
  Eigenvalues}.
\newblock In: J.R. Rice (ed.) {Mathematical Software}, pp. 361 -- 377. Academic
  Press (1977).
\newblock \doi{https://doi.org/10.1016/B978-0-12-587260-7.50018-2}.
\newblock
  \urlprefix\url{http://www.sciencedirect.com/science/article/pii/B9780125872607500182}

\bibitem{GonzalezArroyo:1982ub}
Gonz\'{a}lez-Arroyo, A., Okawa, M.: {A twisted model for large $N$ lattice
  gauge theory}.
\newblock Phys. Lett. \textbf{120B}, 174--178 (1983).
\newblock \doi{10.1016/0370-2693(83)90647-0}

\bibitem{GonzalezArroyo:1982hz}
Gonz\'{a}lez-Arroyo, A., Okawa, M.: {Twisted-Eguchi-Kawai model: A reduced
  model for large-$N$ lattice gauge theory}.
\newblock Phys. Rev. \textbf{D27}, 2397 (1983).
\newblock \doi{10.1103/PhysRevD.27.2397}

\bibitem{Gonzalez-Arroyo:2013bta}
Gonz\'{a}lez-Arroyo, A., Okawa, M.: {Twisted space-time reduced model of large
  $N$ QCD with two adjoint {Wilson} fermions}.
\newblock Phys. Rev. \textbf{D88}, 014514 (2013).
\newblock \doi{10.1103/PhysRevD.88.014514}

\bibitem{doi:10.1007/3-540-28502-4}
Kressner, D.: {Numerical Methods for General and Structured Eigenvalue
  Problems}, \emph{{Lecture Notes in Computational Science and Engineering}},
  vol.~{46}, {1st} edn.
\newblock {Springer-Verlag Berlin Heidelberg} ({2005}).
\newblock \doi{{10.1007/3-540-28502-4}}.
\newblock \urlprefix\url{https://www.springer.com/gp/book/9783540245469}

\bibitem{doi:10.1137/S1064827500366434}
{Mehrmann, Volker. and Watkins, David.}: {Structure-Preserving Methods for
  Computing Eigenpairs of Large Sparse Skew-Hamiltonian/Hamiltonian Pencils}.
\newblock {SIAM J. Sci. Comput.} \textbf{{22}}({6}), {1905--1925} (2001).
\newblock \doi{{10.1137/S1064827500366434}}.
\newblock \urlprefix\url{{ https://doi.org/10.1137/S1064827500366434 }}

\bibitem{Montvay:2001aj}
Montvay, I.: {SUPERSYMMETRIC {YANG}--{MILLS} THEORY ON THE LATTICE}.
\newblock Int. J. Mod. Phys. \textbf{A17}, 2377--2412 (2002).
\newblock \doi{10.1142/S0217751X0201090X}

\bibitem{PETKOVIVANOV1994}
Petkov, M.G., Ivanov, I.G.: {Solution of Symmetric and Hermitian $J$-symmetric
  Eigenvalue Problem}.
\newblock Mathematica Balkanica. New series \textbf{8}, 337 -- 349 (1994).
\newblock
  \urlprefix\url{http://www.math.bas.bg/infres/MathBalk/MB-08/MB-08-337-349.pdf}

\bibitem{ITOSYSTEM}
Research Institute for Information Technology, Kyushu University, ``Introduction of ITO,'' 2018,
\newblock \urlprefix\url{https://www.cc.kyushu-u.ac.jp/scp/eng/system/01_into.html}

\bibitem{doi:10.1142/8229}
Rothe, H.J.: {Lattice Gauge Theories}, 4th edn.
\newblock WORLD SCIENTIFIC (2012).
\newblock \doi{10.1142/8229}.
\newblock \urlprefix\url{https://www.worldscientific.com/doi/abs/10.1142/8229}

\bibitem{SHIMIZU2019372}
Shimizu, N., Mizusaki, T., Utsuno, Y., Tsunoda, Y.: {Thick-restart block
  {Lanczos} method for large-scale shell-model calculations}.
\newblock {Comput. Phys. Commun.} \textbf{244}, 372 -- 384 (2019).
\newblock \doi{https://doi.org/10.1016/j.cpc.2019.06.011}.
\newblock
  \urlprefix\url{http://www.sciencedirect.com/science/article/pii/S0010465519301985}

\bibitem{Stewart:2001:KAL:587707.587809}
Stewart, G.W.: {A {Krylov}--{Schur} Algorithm for Large Eigenproblems}.
\newblock SIAM J. Matrix Anal. Appl. \textbf{23}(3), 601--614 (2001).
\newblock \doi{10.1137/S0895479800371529}.
\newblock \urlprefix\url{http://dx.doi.org/10.1137/S0895479800371529}

\bibitem{WU1999156}
Wu, K., Canning, A., Simon, H., Wang, L.W.: {Thick-Restart {Lanczos} Method for
  Electronic Structure Calculations}.
\newblock {J. Comput. Phys.} \textbf{154}(1), 156 -- 173 (1999).
\newblock \doi{https://doi.org/10.1006/jcph.1999.6306}.
\newblock
  \urlprefix\url{http://www.sciencedirect.com/science/article/pii/S0021999199963064}

\bibitem{doi:10.1137/S0895479898334605}
Wu, K., Simon, H.: {Thick-Restart {Lanczos} Method for Large Symmetric
  Eigenvalue Problems}.
\newblock SIAM J. Matrix Anal. Appl. \textbf{22}(2), 602--616 (2000).
\newblock \doi{10.1137/S0895479898334605}.
\newblock \urlprefix\url{https://doi.org/10.1137/S0895479898334605}

\bibitem{Zhou2008}
Zhou, Y., Saad, Y.: {Block {Krylov}--{Schur} method for large symmetric
  eigenvalue problems}.
\newblock {Numer. Algorithms} \textbf{47}(4), 341--359 (2008).
\newblock \doi{10.1007/s11075-008-9192-9}.
\newblock \urlprefix\url{https://doi.org/10.1007/s11075-008-9192-9}

\end{thebibliography}


\end{document}